\documentclass[11pt,dvipsnames]{amsart}

\usepackage{amsmath,amssymb,amsfonts}
\usepackage[mathscr]{eucal}
\usepackage[margin=1in]{geometry}
\usepackage{url}
\usepackage{multirow}

\usepackage{tikz, tikz-cd}%
\usetikzlibrary{matrix,arrows,decorations.pathmorphing,cd,patterns,calc,backgrounds,patterns}
\tikzset{dummy/.style= {circle,fill,draw,inner sep=0pt,minimum size=1.2mm}}
\tikzset{vertex/.style={fill, circle, minimum size=.1cm, inner sep=0pt}}

\usepackage[
colorlinks, citecolor=black, linkcolor=black, urlcolor=MidnightBlue]{hyperref}
\hypersetup{linktoc=all,} 

\usepackage[dvipsnames]{xcolor}

\definecolor{myblue}{RGB}{100, 143, 225}
\definecolor{mypurple}{RGB}{120, 94, 240}
\definecolor{myred}{RGB}{220, 38, 127}
\definecolor{myorange}{RGB}{254, 97, 0}
\definecolor{myyellow}{RGB}{255, 176, 0}

\usepackage{enumerate}

\numberwithin{equation}{section} 
\numberwithin{figure}{section}
\usepackage[nameinlink,capitalise,noabbrev]{cleveref}

\newcommand{\newrefformat}[2]{}

\crefname{lemma}{Lemma}{Lemmas}
\crefname{theorem}{Theorem}{Theorems}
\crefname{definition}{Definition}{Definitions}
\crefname{proposition}{Proposition}{Propositions}
\crefname{remark}{Remark}{Remarks}
\crefname{corollary}{Corollary}{Corollaries}
\crefname{equation}{Equation}{Equations}
\crefname{construction}{Construction}{Constructions}
\crefname{ex}{Example}{Examples}
\crefname{appsec}{Appendix}{Appendices}
\crefname{subsection}{Section}{Sections}
\crefname{goal}{Goal}{Goals}

\theoremstyle{plain}
\newtheorem{theorem}[equation]{Theorem}

\newtheorem{proposition}[equation]{Proposition}
\newtheorem{lemma}[equation]{Lemma}

\theoremstyle{definition}
\newtheorem{definition}[equation]{Definition}
\newtheorem{example}[equation]{Example}
\newtheorem{remark}[equation]{Remark}
\newtheorem{construction}[equation]{Construction}

\input xy
\xyoption{all}

\newcommand{\Cat}{\mathcal Cat}
\newcommand{\Fin}{\mathscr{F}}
\newcommand{\Fun}{\operatorname{Fun}}

\newcommand{\id}{\operatorname{id}}

\newcommand{\op}{\operatorname{op}}
\newcommand{\T}{\mathcal T}
\newcommand{\tree}{\int_{\Fin_*^{\op}} \T}
\newcommand{\treeg}{\int_{\Fin_{G,*}^{\op}} \T^G}

\title{A comparison of definitions of equivariant trees}
\author[J.E.\ Bergner, M.E.\ Calle, D.\ Chan, A.\ Osorno, and M.\ Sarazola]{Julia E.\ Bergner, Maxine E.\ Calle, David Chan,\\Ang\'elica M. Osorno, and Maru Sarazola}

\address{Department of Mathematics, University of Virginia, Charlottesville, VA 22904}
\email{jeb2md@virginia.edu}

\address{Department of Mathematics, University of Pennsylvania, Philadelphia, PA 19104}
\email{callem@sas.upenn.edu}

\address{Department of Mathematics, Michigan State University, East Lansing, MI 48824}
\email{chandav2@msu.edu}

\address{Department of Mathematics and Statistics, Reed College, Portland, OR 97214}
\email{aosorno@reed.edu}

\address{School of Mathematics, University of Minnesota, Minneapolis MN, 55455}
\email{maru@umn.edu}

\date{\today}

\begin{document}

\begin{abstract}
    We show that various categories of trees can be modeled by Grothendieck constructions on categories of trees with a fixed set of leaves.  We prove this result for the dendroidal category $\Omega$, the category $\Omega^G$ of trees with a $G$-action for a finite group $G$, and finally for the category of genuine equivariant trees $\Omega_G$ that has played an important role in recent work on genuine equivariant operads. 
\end{abstract}

\maketitle

\setcounter{tocdepth}{1} 

\tableofcontents

\section{Introduction}

In recent work, Bonventre and Pereira develop a homotopy theory of genuine equivariant operads \cite{bonventre:19, BonventrePereira, Pereira:EquivariantDendroidalSets}.  Inspired by the $N_\infty$-operads of Blumberg and Hill \cite{BlumbergHillOperadic}, they are designed to encode not only the action of a finite group $G$ but also the more subtle data of norm maps that have proved to be significant in stable equivariant homotopy theory.  Central to their constructions is a certain category $\Omega_G$ of genuine equivariant trees.

Still more recently, in the paper \cite{BBCCM:trees}, we give a definition of equivariant trees and show that they correspond to equivariant partition complexes, generalizing an equivalence in the non-equivariant setting that was established by Robinson \cite{robinson:04} and by Heuts and Moerdijk \cite{HM:23}.  For that work, we were interested in categories $\T^G(A)$ of trees whose leaves are labeled by a fixed $G$-set $A$; in particular, the morphisms in such a category do not affect the set of leaves.

The current paper grew out of a desire to understand the relationship between these two categories of equivariant trees.  Our main result is that the category $\Omega_G$ can be obtained as an iteration of two Grothendieck constructions, one of which takes values in categories closely related to, although not identical to, the categories $\T^G(A)$.

We begin our work non-equivariantly with the dendroidal category $\Omega$, which was first defined by Moerdijk and Weiss \cite{Moerdijk_2007}, and has been used to describe different models for up-to-homotopy operads by Cisinski and Moerdijk \cite{Cisinski_2011}, \cite{Cisinski_2013a}, \cite{Cisinski_2013b}, and Heuts, Hinich, and Moerdijk \cite{HHM:16}.  We prove that the dendroidal category $\Omega$ can be described as the Grothendieck construction of a certain oplax functor from the opposite category of finite pointed sets $\Fin^{\op}_*$ to the category $\Cat$ of small categories.  The following statement is a compilation of \cref{construction:oplaxT}, \cref{prop:functorT}, and \cref{OmegaisTree}.

\begin{theorem} \label{introthm:non eq}
    There is an oplax functor $\T\colon \Fin^{\op}_*\to \Cat$ that takes a finite pointed set $\underline{n}_+ = \{+,1,\dots, n\}$ to a category $\T(\underline{n})$ of trees with $n$ leaves. The dendroidal category $\Omega$ is equivalent to the Grothendieck construction of this oplax functor. 
\end{theorem}

We then establish the corresponding result for an intermediary category $\Omega^G$ of trees with action by a finite group $G$.  More precisely, this category $\Omega^G$ can be obtained as the Grothendieck construction of an oplax functor $\T^G$ that assigns to any finite pointed $G$-set $A_+$ a category $\T^G(A)$ of trees with $G$-action whose leaf set is $G$-equivariantly isomorphic to $A$; see \cref{constr:oplaxTG} and \cref{GOmega is GTree}. 

However, the category $\Omega^G$ is only capable of encoding the structure of operads with a $G$-action; the genuine equivariant category $\Omega_G$ was developed in order to incorporate the structure of norm maps \cite{BonventrePereira, Pereira:EquivariantDendroidalSets}. Indeed, Pereira observes \cite[Proposition 5.47]{Pereira:EquivariantDendroidalSets} that $\Omega_G$ can be modeled as a Grothendieck construction on categories $\Omega^{\underline{G/H}}$, which are equivalent to $\Omega^H$, as $H\leq G$ varies; we provide a full proof in \cref{prop:OmegaGH Grothendieck}. In particular, an object of $\Omega_G$ may be thought of as an $H$-tree for some $H\leq G$. The morphisms of $\Omega_G$ encode both the equivariant maps of trees from $\Omega^H$ for any $H \leq G$ and the maps that change subgroup. In order to make this idea precise, we must pass from considering trees with $G$-action to forests with $G$-action, as detailed in \cref{sec:forests}. 

Finally, we show that $\Omega_G$ may be expressed as an iterated Grothendieck construction; its objects and morphisms are described explicitly in \cref{rmk:unpack OmegaG}. The statement of the following result is morally correct but overlooks some of the technical subtleties; see \cref{constr:oplax TGH}, \cref{prop:TGH oplax}, and \cref{theresult} for the precise details. 

\begin{theorem}\label{introthm:eq}
    The equivariant dendroidal category $\Omega_G$ is equivalent to an iterated Grothendieck construction on the categories $\T^H_+(A)$, for $H\leq G$ and a finite pointed $H$-set $A_+$.
\end{theorem}

As this overview indicates, the Grothendieck construction is a critical tool throughout this paper, making many appearances in different forms.  Acknowledging that readers may have different levels of familiarity with this construction, and especially with the context of oplax functors, we include some essential definitions and necessary results in \cref{sec:cat background}.  We encourage readers either to start with this appendix, or to refer to it as needed, depending on their comfort with this material.

\subsection{Outline}

We begin in \cref{Gdend} by recalling the definition of the dendroidal category $\Omega$ and detailing its generating morphisms. In \cref{sec:Omega is T}, we introduce the oplax functor $\T$ and prove \cref{introthm:non eq}.  Next, \cref{sec:trees with G action} tells the same story in the equivariant setting, introducing the relevant category $\Omega^G$ of trees with $G$-action in \cref{sec:Omega upper G} and modeling it as a Grothendieck construction in \cref{sec:Omega upper G is Groth}. Finally, in \cref{sec:Gtree}, we recall the definition of the equivariant dendroidal category $\Omega_G$. The main results, as summarized in \cref{introthm:eq}, can be found in \cref{sec:Omega lower G is Groth}.  The paper concludes with \cref{sec:cat background}, where we review oplax functors and their Grothendieck construction.

\subsection{Acknowledgments} 
We would like to thank the anonymous referee for a thorough and careful reading and many helpful suggestions. This project was started when JB, MC, DC, and AO were visiting the Isaac Newton Institute for Mathematical Sciences, Cambridge, and thank them for support and hospitality during the programme Equivariant Homotopy Theory in Context; this work was supported by EPSRC grant number EP/R014604/1. Part of the work in this paper took place during our participation in the SQuaREs program at the American Institute of Mathematics, and we are very grateful for their support and for the wonderful work environment that they provided. JB was partially supported by NSF grant DMS-1906281,  DC was partially supported by NSF grant DMS-2135960, MC was partially supported by NSF grant DGE-1845298, AO was partially supported by NSF grant DMS–2204365, and MS was partially supported by NSF grant DMS-2506116.  

\section{Categories of trees} \label{Gdend}

In this section, we focus on categories of non-equivariant trees.  We start by recalling the dendroidal category $\Omega$.  We then show that $\Omega$ can be obtained as a Grothendieck construction of an oplax functor that takes values in categories of trees that have a fixed set of leaves.

\subsection{The category $\Omega$ of trees} \label{recap:dendroidal}

The category of trees is called the \emph{dendroidal category} and denoted by $\Omega$; we refer to the reader to the expository paper \cite{moerdijk} and the detailed treatment in the book \cite{heutsmoerdijk}.  We briefly recall the definition of $\Omega$ here. 

In this paper, we consider trees whose edges need not have vertices at both ends. More precisely, we define a \emph{tree} to be a finite graph with no cycles, where some edges meet only one vertex. We also include the \emph{stick tree} with a single edge and no vertices, denoted by $\eta$. These trees are non-planar, in the sense that the trees do not come with a specified embedding into the plane. 

Every edge in a tree has two ``ends'',  each of which may or may not be attached to a vertex; those that are not are called \emph{loose ends}. Our trees are \emph{rooted}, in the sense that each tree is equipped with the choice of a loose end (which in particular, must exist) called the \emph{root}. The remaining loose ends are called \emph{leaves}. Note that with the exception of the stick tree, the set of loose ends of a tree is in bijection with the set of edges that meet exactly one vertex, which we call \emph{outer edges}. For this reason, we sometimes refer to them interchangeably. 
 
The remaining edges, all of which are attached to two vertices, are called \emph{inner edges}. For example, for any $n\geq 0$, the \emph{corolla} $C_n$ is the tree with one vertex and $n+1$ outer edges, one of which is specified as the root.  Our convention for the stick tree $\eta$ is to consider its unique edge as an outer edge.

For a tree $T$ we denote its sets of vertices, edges, and leaves by $V(T)$, $E(T)$, and $\ell(T)$, respectively.  The choice of root determines a direction to the tree, where the leaves are regarded as inputs and the root is regarded as the output.  Consequently, every vertex $v$ has \emph{incoming edges}, whose number is denoted by $|v|$, the \emph{valence} of the vertex, as well as one \emph{outgoing edge} in the direction of the root.  We do allow \emph{nullary} vertices, namely those that have no incoming edges, and sometimes refer to them as \emph{stumps}.  We also have cause to refer to \emph{unary} vertices, which have a single incoming edge. A vertex is \emph{external} if it meets at most one internal edge.

Following \cite[\S 1.3]{heutsmoerdijk}, we define a \emph{subtree} of a tree $S$ to be a tree $S'$ whose vertices and edges are a subset of those of $S$ such that if a vertex $v$ is in $S'$ then every incoming edge of $v$ is also in $S'$. We write $S'\subseteq S$ to denote that $S'$ is a subtree of $S$.

To define a category of trees, we let a morphism $f\colon T\to S$ consist of a function $E(f)\colon E(T) \to E(S)$ from the set of edges of $T$ to the set of edges of $S$ that preserves the incidence of vertices.  More precisely, for any vertex $v$ of $T$ with incoming edges $e_1,\dots,e_n$ and outgoing edge $e_0$, there is a subtree of $S$, which we denote by $f(v)$, with root $f(e_0)$ and leaves $f(e_1),\dots,f(e_n)$. As a special case, note that if $S'$ is a subtree of $S$, the containment $E(S')\subseteq E(S)$ induces a morphism of trees $S'\to S$ referred as \emph{inclusion}.

\begin{definition}\label{defn: dendroidal cat}
    The \emph{dendroidal category} $\Omega$ is defined to have objects trees and morphisms those functions of edges that preserve the incidence of vertices.
\end{definition}
    
\begin{remark}
    The objects of the category $\Omega$ are often referred to as \emph{operadic} trees, suggestive of the utility of such trees in describing operad structures.  Any tree $T$ generates a free colored operad $\Omega(T)$, and the morphisms $T \rightarrow S$ of $\Omega$ can be defined to be precisely the operad maps $\Omega(T) \rightarrow \Omega(S)$.  This definition of morphisms is more common in the literature, but the one we have given is more useful for our purposes in this paper.
\end{remark}

We can also understand morphisms of trees in terms of a poset structure on the set $E(T)$ of edges of a tree $T$.

\begin{definition}
     For two edges $x,y$ of a tree we write $x\leq y$ if the unique path from $x$ to the root contains $y$.  We say that $x$ and $y$ are \emph{comparable} if either $x\leq y$ or $y\leq x$ and \emph{incomparable} otherwise. 
\end{definition}

Note that if there exists an edge $z$ such that $z\leq x$ and $z\leq y$, then $x$ and $y$ are comparable. Indeed, the unique path $P$ from $z$ to the root must pass through $x$ and $y$, and whichever edge $P$ passes through first is less than the other edge.
   
\begin{proposition}\label{rmk:poset}{{\cite[\S 3.2]{heutsmoerdijk}}}
   If $f\colon T\to S$ is a morphism of trees in $\Omega$ then the underlying map $E(T) \rightarrow E(S)$ of edge sets is a map of posets.  This map of posets entirely determines the morphism of trees, but not every map of posets gives a map of trees. 
\end{proposition} 

We can gain a more precise understanding of the morphisms in $\Omega$ via a description by generators and relations, similar to the description of the simplicial category $\Delta$ in terms of face and degeneracy maps, subject to the cosimplicial identities.  

\begin{definition} Let $S$ be a tree.
    \begin{itemize}
        \item Let $e$ be an inner edge of $S$, and let $T$ be the tree obtained by contracting the edge $e$. That is, $E(T)=E(S)\setminus \{e\}$ and $V(T)$ is obtained by identifying the two vertices adjacent to $e$. The map $\delta_e\colon T \to S$ induced by the inclusion $E(T) \to E(S)$ is called an \emph{inner face map}.  
        
        \item Assume $S$ is not a corolla.  Let $v$ an external vertex of $S$, and let $T$ be the tree obtained by removing $v$ and all outer edges attached to it.  That is, $V(T)=V(S)\setminus\{v\}$ and $E(T)$ is the subset of $E(S)$ consisting of all edges except the outer edges connected to $v$. The map $\delta_v\colon T \to S$ induced by the inclusion of edges $E(T)\to E(S)$ is called an \emph{outer face map}. 

        When $S=C_n$ is a corolla, there is another type of outer face  $\delta_e\colon \eta \to C_n$ given by including the stick tree $\eta$ as an edge $e$ of the corolla $C_n$; note that there are $n+1$ such maps, corresponding to each edge of $C_n$. This case is the only one in which the external vertex $v$ does not determine the map uniquely.  
        
        \item Let $e$ be an edge of $S$, and let $T$ be the tree obtained by placing a vertex $v$ in the middle of $e$ with the effect of splitting $e$ into two edges $e_1$ and $e_2$. The map of edges $E(T)\to E(S)$ that sends both $e_1$ and $e_2$ to $e$ and every other edge to itself induces a map $\sigma_e\colon T \to S$ called a \emph{degeneracy map}. 
    \end{itemize}
\end{definition}

An interesting feature of these maps is that some of them have a second conceptual description in terms of their opposites.  An outer face map $\delta_v$ can be regarded as the opposite of the removal of the external vertex $v$ and its incoming leaves, a process sometimes referred to as \emph{pruning}.  A degeneracy map $\sigma_e$ can also be described as the opposite of splitting an edge $e$ into two edges with a unary vertex between them. 

\begin{example}
    An example of an inner face map is the following, induced by the contraction of the blue edge.
    \[
    \scalebox{0.85}{
    \begin{tikzpicture} 
    [level distance=10mm, 
    every node/.style={fill, circle, minimum size=.1cm, inner sep=0pt}, 
    level 1/.style={sibling distance=20mm}, 
    level 2/.style={sibling distance=20mm}, 
    level 3/.style={sibling distance=14mm},
    level 4/.style={sibling distance=7mm}]
    \node (tree) at (-3,0) [style={color=white}] {} [grow'=up] 
    child {node (level2) {} 
	child{ node (level3) {}
		child
		child
	}
	child
	child
    };
    \node[fill=white, label=$\longrightarrow$] at (-0.5,1) {};
    \node (tree) at (3,0) [style={color=white}] {} [grow'=up] 
    child {node (level1) {} 
	child{ node {}
		child
		child
	}
	child [draw=myblue, thick]{ node (level2) {}
		child [draw=black]
            child [draw=black]
	}
    };

    \tikzstyle{every node}=[]
    \end{tikzpicture}
    }
    \] 
\end{example}

\begin{example}
    In contrast, an example of an outer face map is given by the inclusion of the subtree
    \[
    \scalebox{0.85}
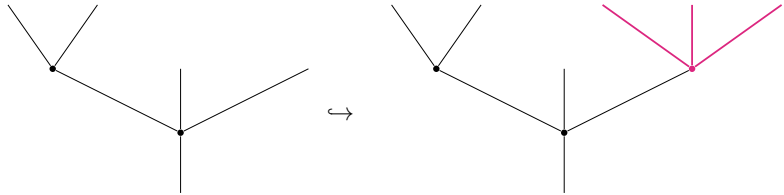
{
    \begin{tikzpicture} 
    [level distance=10mm, 
    every node/.style={fill, circle, minimum size=.1cm, inner sep=0pt}, 
    level 1/.style={sibling distance=20mm}, 
    level 2/.style={sibling distance=20mm}, 
    level 3/.style={sibling distance=14mm},
    level 4/.style={sibling distance=7mm}]
    \node (tree) at (-3,0) [style={color=white}] {} [grow'=up] 
    child {node (level2) {} 
	child{ node (level3) {}
		child
		child
	}
	child
	child
    };
    \node[fill=white, label=$\hookrightarrow$] at (-0.5,1) {};
    \node (tree) at (3,0) [style={color=white}] {} [grow'=up] 
    child {node (level1) {} 
	child{ node {}
		child
		child
	}
        child
        child{ node[color=myred] {}
    		child[draw=myred, thick]
    		child[draw=myred, thick]
                child[draw=myred, thick]
    	}
    };

    \tikzstyle{every node}=[]
    \end{tikzpicture}
    }
    \]
    below the red vertex and incoming leaves. Alternatively, it is induced by the removal of the red vertex and its incoming leaves. 
\end{example}

\begin{example}\label{ex:root graft}
    Another kind of outer face map is the inclusion of a subtree in a way that does not preserve the root, as depicted below.
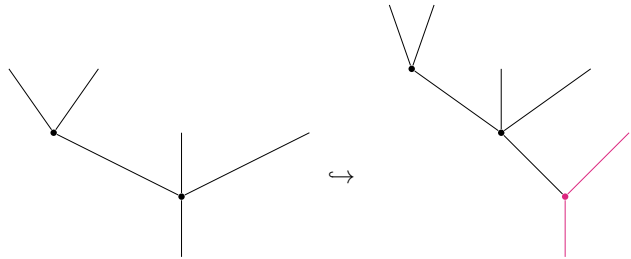
    \[
    \scalebox{0.85}{
    \begin{tikzpicture} 
    [level distance=10mm, 
    every node/.style={fill, circle, minimum size=.1cm, inner sep=0pt}, 
    level 1/.style={sibling distance=20mm}, 
    level 2/.style={sibling distance=20mm}, 
    level 3/.style={sibling distance=14mm},
    level 4/.style={sibling distance=7mm}]
    \node (tree) at (-3,0) [style={color=white}] {} [grow'=up] 
    child {node (level2) {} 
	child{ node (level3) {}
		child
		child
	}
	child
	child
    };
    \node[fill=white, label=$\hookrightarrow$] at (-0.5,1) {};
    \node (tree) at (3,0) [style={color=white}] {} [grow'=up] 
    child[draw = myred] {node[color=myred] (level1) {} 
	child[draw = black]{ node {}
		child{ node {}
            child
            child}
		child
        child
	}
        child
    };

    \tikzstyle{every node}=[]
    \end{tikzpicture}
    }
    \] 
    In this case the root is mapped to an inner edge.  In other words, this picture is the result of grafting the original tree onto one of the leaves of the corolla $C_2$ onto the root.  Alternatively, we can think of this map as induced by a pruning that removes the root and its adjacent vertex.
\end{example}

\begin{example}
    An example of a degeneracy map is below.
    \[
    \scalebox{0.85}{
    \begin{tikzpicture} 
    [level distance=10mm, 
    every node/.style={fill, circle, minimum size=.1cm, inner sep=0pt}, 
    level 1/.style={sibling distance=20mm}, 
    level 2/.style={sibling distance=20mm}, 
    level 3/.style={sibling distance=14mm},
    level 4/.style={sibling distance=7mm}]
    \node (tree) at (-3,0) [style={color=white}] {} [grow'=up] 
    child {node (level2) {} 
	child{ node (level3) {}
		child
		child
	}
	child[color=myorange]{ node[color=myorange] (level3) {}
            child
        }
        child
    };

    \node[fill=white, label=$\longrightarrow$] at (-0.5,1) {};

    \node (tree) at (3,0) [style={color=white}] {} [grow'=up] 
    child {node (level2) {} 
	child{ node (level3) {}
		child
		child
	}
	child[color=myorange]
    child
    };

    \tikzstyle{every node}=[]
    \end{tikzpicture}
    }
    \]
    It is induced by the splitting of the orange edge in the target. It can also be thought of as the removal of the indicated unary vertex, merging the two edges.
\end{example}

A map $T \rightarrow S$ of trees is an \emph{isomorphism} if the corresponding map $E(T)\to E(S)$ a bijection. Note that an isomorphism also induces a bijection on the sets of vertices. For more details on this definition, see \cite[\S3.1]{heutsmoerdijk}.

\begin{example}
    An example of an isomorphism of trees is given by
    \[
    \scalebox{0.85}
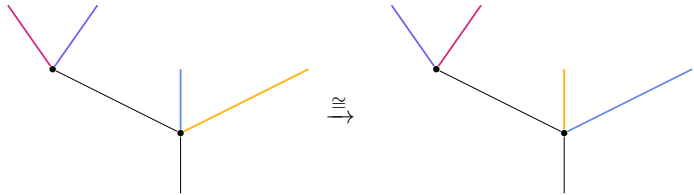
{
    \begin{tikzpicture} 
    [level distance=10mm, 
    every node/.style={fill, circle, minimum size=.1cm, inner sep=0pt}, 
    level 1/.style={sibling distance=20mm}, 
    level 2/.style={sibling distance=20mm}, 
    level 3/.style={sibling distance=14mm},
    level 4/.style={sibling distance=7mm}]
    \node (tree) at (-3,0) [style={color=white}] {} [grow'=up] 
    child {node (level2) {} 
	child{ node (level3) {}
		child[draw = myred, thick]
		child[draw = mypurple, thick]
	}
	child[draw = myblue, thick]
	child[draw = myyellow, thick]
    };
    \node[fill=white, label=$\xrightarrow{\cong}$] at (-0.5,1) {};
    \node (tree) at (3,0) [style={color=white}] {} [grow'=up] 
    child {node (level2) {} 
	child{ node (level3) {}
		child[draw = mypurple, thick]
		child[draw = myred, thick]
	}
	child[draw = myyellow, thick]
	child[draw = myblue, thick]
    };
    \end{tikzpicture}
    } 
    \] 
    where the corresponding colors indicate the bijection between the leaves.
\end{example}

The morphisms in $\Omega$ are generated by the inner face maps, outer face maps, and degeneracy maps, as well as isomorphisms, subject to certain relations.  We refer the reader to \cite[\S 3.3.4]{heutsmoerdijk} or \cite{moerdijk} for more details, including the list of the codendroidal identities, and simply state the following result.

\begin{proposition}[{\cite[3.9, 3.10]{heutsmoerdijk}}] \label{istrue} 
    Any morphism $f \colon T \rightarrow S$ factors uniquely (up to unique isomorphism) as $\delta_o \circ \delta_i \circ \zeta \circ \sigma$, where $\sigma$ is a composite of degeneracy maps, $\zeta$ is an isomorphism, $\delta_i$ is a composite of inner face maps, and $\delta_o$ is a composite of outer face maps.
\end{proposition}  

\subsection{The category $\Omega$ as a Grothendieck construction}\label{sec:Omega is T}

In this section, we show how the dendroidal category $\Omega$ can be alternatively described as the Grothendieck construction of a certain oplax functor.   We encourage readers unfamiliar with oplax functors or the Grothendieck construction to refer to \cref{sec:cat background} before proceeding.

To describe this functor, we start by introducing a category of trees whose leaves are labeled in the sense of the following definition.  For a tree $T$, recall that $\ell(T)$ denotes the set of leaves.

\begin{definition}\label{def:labeling}
    Let $X$ be a set. A tree with \emph{leaves labeled by $X$} is a pair $(T,\chi)$ consisting of a tree $T$ and a bijection $\chi_T\colon X\to \ell(T)$, called the \textit{labeling}. When $X=\underline{n}$, we refer to the leaf labeled by $\chi_T(i)$ as the \textit{$i^{th}$ leaf} of $T$, for $i\in \underline{n}$. 
    A morphism $f\colon (T,\chi_T)\to (S,\chi_S)$ of trees with leaves labeled by $X$ is \textit{label-preserving} if it sends leaves to leaves, preserves the root, and $\ell(f)\circ \chi_T = \chi_S$. Note, in particular, that $\ell(f)$ is a bijection. 
\end{definition}

We often abuse notation and simply write $T$, leaving the labeling implicit.

\begin{definition}\label{defn:T(n)} 
    For any $n \geq 0$, let $\T(\underline{n})$ be the category of trees with leaves labeled by $\underline{n}=\{1, 2, \ldots, n\}$ and whose morphisms are label-preserving morphisms.
\end{definition}

\begin{remark}
    Note that because outer face maps either change the set of leaves, or do not preserve the root, of a tree, they are not morphisms in $\T(\underline{n})$.  In particular morphisms in $\T(\underline{n})$ are generated by inner faces, degeneracies and isomorphisms, subject to the relations between them in $\Omega$.  
\end{remark}

\begin{example}
    The stump tree, which consists of a root edge and a vertex, is in the category $\T(\underline{0})$, where the labeling bijection is just the identity on the empty set. This category also contains other trees with no leaves whose labeling bijections are also the identity on the empty set.
\end{example}

\begin{example}
    The stick tree $\eta$, which has no vertices and a single edge, is in $\T(\underline{1})$. It has two loose ends, one which is the root, and the other is labeled by $1$. We note that there are no maps between $\eta$ and any other object of $\T(\underline{1})$, since maps in $\T(\underline{1})$ necessarily preserve both the leaves and the root.
\end{example}

\begin{remark}
    The category $\T(\underline{n})$ is not quite the same as the one described in \cite{BBCCM:trees}.  There, the corollas, including $\eta$, were omitted in order to avoid having a contractible nerve.  Furthermore, because trees were required to be reduced, it was not possible to have degeneracy maps.  Thus, the category we have described here is an appropriate expansion that includes all trees with $n$ leaves and morphisms that preserve the labeling on the leaves.
\end{remark}

\begin{remark}
    We note that the set of isomorphism classes of objects in $\T(\underline{n})$, equipped with the natural $\Sigma_n$-action, form the operad of trees considered in \cite[Example 1.5]{heutsmoerdijk}.
\end{remark}

Let $\Fin_*$ denote the skeletal category of finite pointed sets whose objects are given by $\underline{n}_+=\{+, 1, \ldots, n\}$ for $n \geq 0$ and whose morphisms are pointed functions.  We want to assign to every $\underline{n}_+$ the corresponding category $\T(\underline{n})$ in such a way that we can obtain $\Omega$ as a Grothendieck construction.  Let us first give an explicit description of this assignment. 

\begin{construction}\label{construction:oplaxT}
    Define the assignment $\T$ from the category $\Fin_*^{\op}$ to the category $\Cat$ of small categories and functors as follows.
    \begin{itemize}
        \item On objects, $\T$ maps $\underline{n}_+$ to the category $\T(\underline{n})$ of \cref{defn:T(n)}.
        
        \item On morphisms, $\T$ maps a function $\varphi \colon \underline{m}_+ \to \underline{n}_+$ in $\Fin_*$ to the functor $\T(\varphi)=\varphi^*\colon \T(\underline{n})\to \T(\underline{m})$ as follows.        
        \begin{itemize} 
            \item For a tree $T$ in $\T(\underline{n})$, construct $\varphi^*(T)$ by attaching a corolla labeled by $\varphi^{-1}(i)$ to the leaf in $T$ labeled by $i$ for each $i\in \underline{n}_+$.  In particular, if $\varphi^{-1}(i)=\varnothing$, the leaf in $T$ labeled by $i$ becomes a stump. Note that this includes the case in which $i=+$, where we attach the corolla labeled by $\varphi^{-1}(+)$ to the root. The new root is the edge labeled by $+\in\varphi^{-1}(+)$. 
            \item A morphism $f\colon T\to S$ in $\T(\underline{n})$ is sent to the morphism $\varphi^*(f)\colon\varphi^*(T)\to\varphi^*(S)$ given by $\varphi^*(f)(i)=i$ for all leaves $i\in\underline{m}$, and by $\varphi^*(f)(e)=f(e)$ for all edges $e$ previously present in $T$. 
        \end{itemize} 
    \end{itemize}
\end{construction}

\begin{example}
    We give an example to highlight what $\varphi^*$ does on $i=+$. Let $\varphi\colon \underline{3}_+\to \underline{2}_+$ be the function with $\varphi(1) = \varphi(2)=\varphi(+) = +$ and $\varphi(3) = 1$.  Then, if $T\in \T(\underline{2})$ denotes the corolla depicted below left, the tree $\varphi^*(T)$ is as depicted on the right of \cref{phiplus}.
    \begin{figure}[h!]
    \[\scalebox{0.85}{
        \begin{tikzpicture} 
        [level distance=10mm, 
        every node/.style={fill, circle, minimum size=.1cm, inner sep=0pt}, 
        level 1/.style={sibling distance=20mm}, 
        level 2/.style={sibling distance=20mm}, 
        level 3/.style={sibling distance=14mm},
        level 4/.style={sibling distance=7mm}]
        \node (tree) at (-3,0) [style={color=white}] {} [grow'=up] 
        child {node (x) {} 
	   child{   {}}
	   child { {}}
        };
        \node[fill=white, label=$T$] at (-3,-0.5) {};
        \node[fill=white, label=$\varphi^*(T)$] at (3,-1) {};
        \node (tree) at (3,0) [style={color=white}] {} [grow'=up] 
        child {node (y) {} 
        child { {}}
        child { {}}
        child {node  {} 
	    child{  node {}
        child { {}}}
	    child{  node {}
        }
        }
         };
        \tikzstyle{every node}=[]

    \node at ($(x) + (10mm, 12mm)$){$2$};
    \node at ($(x) + (-10mm, 12mm)$){$1$};
    \node at ($(y) + (13mm, 32mm)$){$3$};
    \node at ($(y) + (0mm, 12mm)$){$2$};
    \node at ($(y) + (-20mm, 12mm)$){$1$};
        \end{tikzpicture}
        } \]
        \caption{Example of $\varphi^*(T)$}\label{phiplus}
    \end{figure}
\end{example}

\begin{example}
    For the unique pointed function $\phi\colon \underline{0}_+\to \underline{n}_+$, the construction above sends a tree $T$ with leaves labeled by $\underline{n}$ to the tree obtained by attaching nullary vertices to the leaves of $T$ and prolonging the root. This construction is depicted in \cref{fig:i+} for $T$ a corolla and $n=3$. 
    \begin{figure}[h!]
    \[\scalebox{0.85}{
        \begin{tikzpicture} 
        [level distance=10mm, 
        every node/.style={fill, circle, minimum size=.1cm, inner sep=0pt}, 
        level 1/.style={sibling distance=20mm}, 
        level 2/.style={sibling distance=20mm}, 
        level 3/.style={sibling distance=14mm},
        level 4/.style={sibling distance=7mm}]
        \node (tree) at (-3,0) [style={color=white}] {} [grow'=up] 
        child {node (x) {} 
	   child{   {}}
	   child { {}}
      child { {}}
        };
        \node[fill=white, label=$T$] at (-3,-0.5) {};
        \node[fill=white, label=$\varphi^*(T)$] at (3,-1) {};
        \node (tree) at (3,0) [style={color=white}] {} [grow'=up] 
        child { node {}
        child {node (y) {} 
        child { node {}}
        child { node {}}
        child {node  {}}
        }
         };
        \tikzstyle{every node}=[]

    \node at ($(x) + (20mm, 12mm)$){$1$};
    \node at ($(x) + (0mm, 12mm)$){$2$};
        \node at ($(x) + (-20mm, 12mm)$){$3$};
        \end{tikzpicture}
        } \]
        \caption{Example of $\varphi^*(T)$}
        \label{fig:i+}
    \end{figure}
\end{example}

A naive hope would be that this assignment $\underline{n}_+ \mapsto \T(\underline{n})$ defines a contravariant functor of which $\Omega$ is the Grothendieck construction.  Unfortunately, this construction is not quite a functor, as illustrated by the following explicit examples. 

\begin{example}
    Let us first give an example to illustrate why $\T$ does not preserve identity morphisms. Consider the identity function $\id_{\underline{2}_+} \colon \underline{2}_+ \to \underline{2}_+$, and let $T$ be the tree of $\T(\underline{2})$ depicted to the left in \cref{fig:idnotfunctorial}.

    \begin{figure}[h!]
    \[\scalebox{0.85}{
        \begin{tikzpicture} 
        [level distance=10mm, 
        every node/.style={fill, circle, minimum size=.1cm, inner sep=0pt}, 
        level 1/.style={sibling distance=20mm}, 
        level 2/.style={sibling distance=20mm}, 
        level 3/.style={sibling distance=14mm},
        level 4/.style={sibling distance=7mm}]
        \node (tree) at (-3,0) [style={color=white}] {} [grow'=up] 
        child {node (x) {} 
	   child{   {}}
	   child { {}}
        };
        \node[fill=white, label=$T$] at (-3,-0.5) {};
        \node[fill=white, label=$\id^*_{\underline{2}_+}(T)$] at (3,-1) {};
        \node (tree) at (3,0) [style={color=white}] {} [grow'=up] 
        child {node  {} 
        child {node (y) {} 
	    child{  node {}
        child { {}}}
	    child{  node {}
        child { {}}}
        }
         };
        \tikzstyle{every node}=[]

        \node at ($(x) + (10mm, 12mm)$){$2$};
        \node at ($(x) + (-10mm, 12mm)$){$1$};
        \node at ($(y) + (7mm, 22mm)$){$2$};
        \node at ($(y) + (-7mm, 22mm)$){$1$};
        \end{tikzpicture}
        } \]
        \caption{Non-functoriality: the identity}
        \label{fig:idnotfunctorial}
    \end{figure}
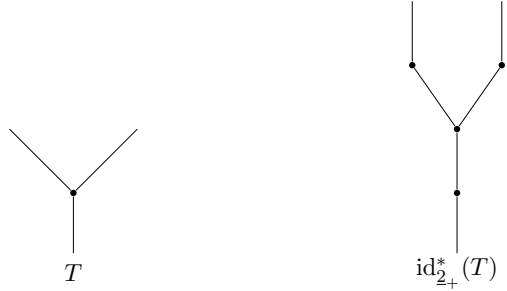
    
    According to \cref{construction:oplaxT}, the functor $\id_{\underline{2}_+}^* \colon \T(\underline{2}) \to\T(\underline{2})$ acts on $T$ by attaching a 1-corolla to each of the leaves and the root of $T$, resulting in the tree $\id_{\underline{2}_+}^*(T)$ depicted to the right in \cref{fig:idnotfunctorial}.  Since these trees are not the same, we can conclude that $\id_{\underline{2}_+}^*\neq\id_{\T(\underline{2})}$; hence $\T$ does not preserve identity morphisms. 

    The assignment $\T$ also fails to preserve composition. As an example, consider the pointed functions $\underline{3}_+ \xrightarrow{\gamma} \underline{5}_+ \xrightarrow{\varphi} \underline{4}_+$ given by $\gamma(1)=\gamma(2)=1$, $\gamma(3)=4$; and $\varphi(1)=1$, $\varphi(2)=\varphi(3)=2$, $\varphi(4)=\varphi(5)=3$. Consider the following tree $T$ in $\T(\underline{4})$.
    \[
    \scalebox{0.85}
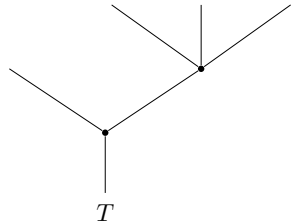
{
    \begin{tikzpicture} 
    [level distance=10mm, 
    every node/.style={fill, circle, minimum size=.1cm, inner sep=0pt}, 
    level 1/.style={sibling distance=20mm}, 
    level 2/.style={sibling distance=30mm}, 
    level 3/.style={sibling distance=14mm},
    level 4/.style={sibling distance=7mm}]
    \node (tree) at (0,0) [style={color=white}] {} [grow'=up] 
    child {node (x) (level2) {} 
	child{  (level3) {}
	}
	child {node (y) {} 
            child { {} }
            child { {} }
            child { {} }
            }	
    };
    \node[fill=white, label=$T$] at (0,-0.5) {};
  
    \tikzstyle{every node}=[]
    \node at ($(x) + (14mm, 12mm)$){$1$};
    \node at ($(y) + (-14mm, 12mm)$){$2$};
    \node at ($(y) + (0mm, 12mm)$){$3$};
    \node at ($(y) + (14mm, 12mm)$){$4$};
    \end{tikzpicture}
    }
    \]
    Since $\varphi \circ \gamma(1) = \varphi \circ \gamma(2)=1$, $\varphi\circ\gamma(3)=3$, we have that $(\varphi\circ\gamma)^*(T)$ is obtained from $T$ by attaching a 2-corolla to the first leaf, a 1-corolla to the third leaf and the root, and stumps to the second and fourth leaves. The result is the tree depicted on the left in \cref{fig:compnotfunctorial}, where the added nodes and edges are colored in orange.
    
    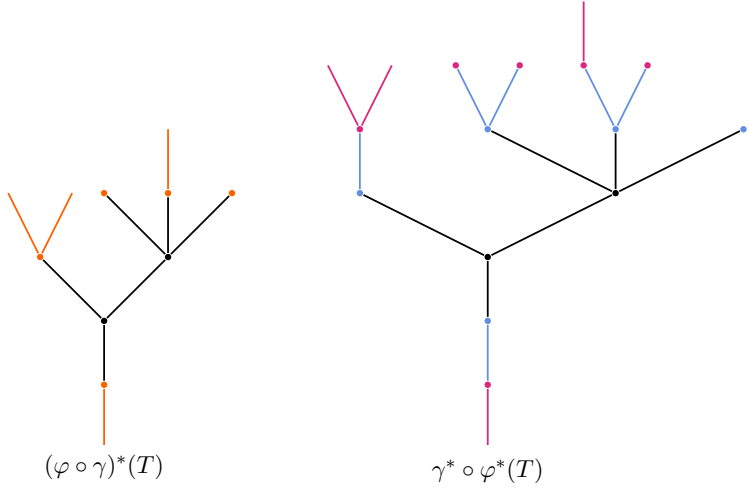
\begin{figure}[h!]
    \[
    \scalebox{0.85}{
    \begin{tikzpicture} 
    [level distance=10mm, 
    every node/.style={fill, circle, minimum size=.1cm, inner sep=0pt}, 
    level 1/.style={sibling distance=60mm}, 
    level 2/.style={sibling distance=40mm}, 
    level 3/.style={sibling distance=20mm},
    level 4/.style={sibling distance=10mm}]
    \node (tree) at (-3,0) [style={color=white}] {} [grow'=up] 
    child[draw=myorange, thick] {node[color=myorange]  {} 
    child[draw=black] {node (x) (level2) {} 
	child{ node[color=myorange] (level3) {}
		child [draw=myorange, thick]
		child [draw=myorange, thick]
	}
	child {node {} 
            child {node[color=myorange] {} }
            child {node[color=myorange] {} child [draw=myorange, thick]}
            child {node[color=myorange] {} }
            }	
        }
    };
    \node[fill=white, label=$(\varphi\circ\gamma)^*(T)$] at (-3,-1.3) {};
    \node[fill=white, label=$\gamma^*\circ\varphi^*(T)$] at (3,-1.3) {};
    \node (tree) at (3,0) [style={color=white}] {} [grow'=up]
    child [draw=myred, thick] {node[color=myred] {}
    child [draw=myblue, thick] {node[color=myblue] {}
    child [draw=black] {node (y) {} 
	child [level 2] { node[color=myblue] {}
            child [draw=myblue, thick, level 4] {node[color=myred] {}
		child [draw=myred, thick]
		child [draw=myred, thick]}
	}
	child [level 2] {node {} 
            child [level 3] {node[color=myblue] {}  
                child [draw=myblue, thick, level 4] {node[color=myred] {} }
                child [draw=myblue, thick, level 4] {node[color=myred] {}}
                }
            child [level 3] {node[color=myblue] {} 
                child [draw=myblue, thick, level 4] {node[color=myred] {}
                    child [draw=myred, thick]}
                child [draw=myblue, thick, level 4] {node[color=myred] {}}
                }
            child [level 3] {node[color=myblue] {} }
            }	
        }
        }
    };
    \tikzstyle{every node}=[]
    \node at ($(x) + (-15mm, 32mm)$){$1$};
    \node at ($(x) + (-5mm, 32mm)$){$2$};
    \node at ($(x) + (10mm, 42mm)$){$3$};
    \node at ($(y) + (-25mm, 32mm)$){$1$};
    \node at ($(y) + (-15mm, 32mm)$){$2$};
    \node at ($(y) + (15mm, 42mm)$){$3$};
        
    \end{tikzpicture}
    }
    \]
    \caption{Non-functoriality: composition}
        \label{fig:compnotfunctorial}
    \end{figure}
    
    On the other hand, the tree $\gamma^*\circ\varphi^*(T)$ is obtained from $T$ by first attaching corollas according to the action of $\varphi^*$, which is depicted in blue in the tree on the right in \cref{fig:compnotfunctorial}, and then attaching corollas according to $\gamma^*$, depicted in red. Again, we observe from this figure that $(\varphi\circ\gamma)^*(T)\neq \gamma^*\circ\varphi^*(T)$; hence $\T$ does not preserve composition.    
\end{example}

Although the assignment $\T\colon \Fin_*^{\op} \to \Cat$ does not strictly preserve composites and identities, the next result shows that it gives rise to an oplax functor, which is sufficient for our eventual goal of considering its Grothendieck construction. 

\begin{proposition}\label{prop:functorT}
    The assignment $\T\colon \Fin_*^{\op} \to \Cat$ of \cref{construction:oplaxT} defines an oplax functor.
\end{proposition}

\begin{proof}
    To show that $\T$ is an oplax functor, we must provide the data of natural transformations 
    \[\tau_n\colon (\id_{\underline{n}_+})^* \Rightarrow \id_{\T(\underline{n})} \quad \text{and} \quad \tau_{\gamma,\varphi}\colon(\varphi \circ \gamma)^* \Rightarrow \gamma^* \circ \varphi^*\]
    for all $\underline{n}_+$ and all maps $\underline{\ell}_+ \xrightarrow{\gamma} \underline{m}_+ \xrightarrow{\varphi} \underline{n}_+$ in $\Fin_*$, and then show that they satisfy the coherence conditions from \cref{defn:oplax}. The coherence conditions follow from the codendroidal identities that relate the generating morphisms of $\Omega$; see \cite[3.3.4]{heutsmoerdijk}.

    To define $\tau_n$, given a tree $T$ in $\T(\underline{n})$, observe that for the identity map, the tree $\id_{\underline{n}_+}^*(T)$ is obtained from $T$ by attaching a single edge to each leaf and the root; an example of this process can be seen in \cref{fig:idnotfunctorial}. Thus, $(\tau_n)_T\colon \id_{\underline{n}_+}^*(T)\to T$ is given by the composite of the degeneracies that remove the unary vertices that were added.  The naturality of $\tau_n$ with respect to the inner face maps follows from the compatibility of inner face maps (which only affect inner edges) and degeneracies which eliminate unary vertices connected to leaves. Naturality with respect to the degeneracy maps is a consequence of the codendroidal identity \cite[3.3.4(5)]{heutsmoerdijk}. It follows that $\tau_n$ is a natural transformation, as any morphism can be written as a composite of inner faces, degeneracies, and isomorphisms.

    Next, given a composite $\underline{\ell}_+ \xrightarrow{\gamma} \underline{m}_+ \xrightarrow{\varphi} \underline{n}_+$ in $\Fin$ and a tree $T$ in $\T(\underline{\ell})$, to define $\tau_{\gamma,\varphi}$ we start by describing the tree $\gamma^*\circ\varphi^*(T)$. This tree is obtained from $T$ by first attaching a corolla labeled by $\varphi^{-1}(i)$ to the leaf (or root) in $T$ labeled by $i$ for each $i\in\underline{n}_+$. Then, to each leaf (or root) $j$ in the corolla $\varphi^{-1}(i)$, we attach a new corolla labeled by $\gamma^{-1}(j)$; see \cref{fig:compnotfunctorial} for an example. In particular, note that the second part of this process turns each edge in the corollas added by $\varphi^*$  into an inner edge. The map $(\tau_{\gamma, \varphi})_T\colon (\varphi\circ\gamma)^*(T)\to \gamma^*\circ\varphi^*(T)$ is given by the composite of the inner faces corresponding to the contraction of all these newly created inner edges. In the picture of \cref{fig:compnotfunctorial}, this map corresponds to contracting all blue edges.  Naturality of $\tau_{\gamma,\varphi}$ follows from the codendroidal relations that hold for composites of inner faces and of inner faces with degeneracies \cite[3.3.4(2) and (9)]{heutsmoerdijk}, and the compatibility of $\tau_{\gamma, \varphi}$ with isomorphisms.

    It remains to check that these natural transformations satisfy the required compatibilities encoding unitality and associativity. For unitality, we must check that for any $\varphi\colon\underline{m}_+\to\underline{n}_+$, the triangles 
    \begin{equation}\label{eq:unitality}
        \begin{tikzcd}[column sep=small]
            \varphi^*=(\varphi\circ\id_{\underline{m}_+})^* \ar[r, equal]\ar[d, Rightarrow,"\tau_{\id,\varphi}"'] &  \varphi^*=\id_{\T(\underline{m})}\circ\varphi^*\\
            \id_{\underline{m}_+}^*\circ\varphi^*\ar[ur, Rightarrow, "\tau_{m}\cdot \varphi^*"']
        \end{tikzcd}\qquad 
        \begin{tikzcd}[column sep=small]
            \varphi^*=(\id_{\underline{n}_+}\circ\varphi)^* \ar[r, equal]\ar[d, Rightarrow,"\tau_{\varphi, \id}"'] &  \varphi^*=\varphi^*\circ \id_{\T(\underline{n})}\\
            \varphi^*\circ\id_{\underline{n}_+}^*\ar[ur, Rightarrow, "\varphi^*\cdot\tau_{n}"']
        \end{tikzcd}
    \end{equation}
    commute, where $\tau_{m}\cdot \varphi^*$ denotes the whiskering of the natural transformation $\tau_{m}$ with the functor $\varphi^*$; see \cref{defn:whiskering}.   Applying the definitions of these natural transformations, we find that the commutativity of these triangles is due to the fact that the composite
    \[ \scalebox{0.85}
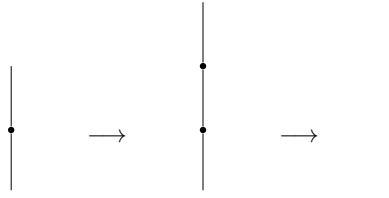
{
    \begin{tikzpicture} 
        [level distance=10mm, 
        every node/.style={fill, circle, minimum size=.1cm, inner sep=0pt}, 
        level 1/.style={sibling distance=20mm}, 
        level 2/.style={sibling distance=30mm}, 
        level 3/.style={sibling distance=14mm},
        level 4/.style={sibling distance=7mm}]
        \node (tree) at (-3,0) [style={color=white}] {} [grow'=up] 
        child {node  {} 
	       child
	       };
        \node[fill=white, label=$\longrightarrow$] at (1.5,0.5) {};
        \node[fill=white, label=$\longrightarrow$] at (-1.5,0.5) {};
        \node (tree) at (0,0) [style={color=white}] {} [grow'=up] 
        child {node  {} 
	       child {node {} 
            child}
    	};
        \node (tree) at (3,0) [style={color=white}] {} [grow'=up] 
        child {node  {} 
	       child};
        \tikzstyle{every node}=[]
    \end{tikzpicture} } \]
    of the inner face corresponding to contracting the edge in the middle, followed by the degeneracy that removes one of the unary vertices, is equal to the identity.

    The associativity condition  requires that, for each composite 
    \[ \underline{k}_+ \xrightarrow{\delta} \underline{l}_+ \xrightarrow{\gamma} \underline{m}_+ \xrightarrow{\varphi} \underline{n}_+\] 
    in $\Fin_*$, the diagram 
    \begin{equation}\label{eq:associativity}\begin{tikzcd}
        (\varphi\circ\gamma\circ\delta)^*\rar[Rightarrow, "\tau_{\gamma\circ\delta, \varphi}"]\dar[Rightarrow,"\tau_{\delta, \varphi\circ\gamma}"'] & (\gamma\circ\delta)^*\circ\varphi^*\dar[Rightarrow,"\tau_{\delta,\gamma}\cdot\varphi^*"] \\
        \delta^*\circ (\varphi\circ\gamma)^*\rar[Rightarrow, "\delta^*\cdot\tau_{\gamma, \varphi}"'] & \delta^*\circ\gamma^*\circ\varphi^*
    \end{tikzcd}\end{equation}
    commutes, which follows directly from the relations regarding composites of inner faces.
\end{proof}

Now, we can apply the Grothendieck construction from \cref{defn:grothconstruction} to the oplax functor from \cref{prop:functorT} to obtain a category we denote by $\tree$. The objects of $\tree$ consist of pairs $(\underline{n}_+,T)$ where $T$ is an object in $\T(\underline{n})$. A morphism $(\underline{n}_+,T) \to (\underline{m}_+, S)$ is given by a pair $(\varphi,f)$, consisting of a pointed function $\varphi\colon \underline{m}_+ \to \underline{n}_+$ in $\Fin_*$, and a morphism $f\colon \varphi^*(T) \to S$ in $\T(\underline{m})$. The identity morphism $\id_{(\underline{n}_+, T)} \colon (\underline{n}_+,T) \to (\underline{n}_+, T)$ is given by $(\id_{\underline{n}_+},(\tau_{n})_T)$. Given maps
\[(\underline{n}_+,T) \xrightarrow{(\varphi,f)} (\underline{m}_+, S) \xrightarrow{(\gamma,g)} (\underline{\ell}_+,R),\]
the composite is given by $\varphi\circ \gamma$ in the first coordinate, and by the composite 
\[ (\varphi\circ \gamma)^* (T) \xrightarrow{(\tau_{\gamma, \varphi})_T} \gamma^* \varphi^* (T) \xrightarrow{\gamma^* (f)} \gamma^* (S) \xrightarrow{g} R\]
in the second coordinate.

The fact that $\id_{(\underline{n}_+,T)}$ is actually an identity follows from unitality \eqref{eq:unitality}, and the fact that composition is associative for a string of composable maps $(\varphi,f), (\gamma,g), (\delta,h)$ follows from the naturality of $\tau_{\delta,\gamma}$ together with the associativity condition on $\tau$ \eqref{eq:associativity}. 

\begin{theorem}\label{OmegaisTree}
    There is an equivalence of categories $F \colon \tree \xrightarrow{\simeq} \Omega$. 
\end{theorem}

\begin{proof}
    We define $F$ on objects by $F(\underline{n}_+,T)=T$.  Recall that a morphism $(\varphi, f) \colon (\underline{n}_+, T) \to (\underline{m}_+, S)$ in $\tree$ is given by a pointed function $\varphi\colon\underline{m}_+\to\underline{n}_+$ together with a map $f\colon \varphi^*(T)\to S$ in $\T(\underline{m})$, where $\varphi^*(T)$ is obtained from $T$ by attaching corollas as described in \cref{construction:oplaxT}.  There is a unique map $\iota_{\varphi,T}\colon T\to\varphi^*(T)$ that can be obtained as a composite of outer face maps in $\Omega$; while the decomposition into outer face maps itself is not unique, as these maps may be permuted, the composite $\iota_{\varphi,T}$ is unique. We define $F(\varphi, f)$ as the composite
    \[T \xrightarrow{\iota_{\varphi,T}} \varphi^* (T)\xrightarrow{f} S.\]

    To check that $F$ preserves identities, observe that $F(\id_{\underline{n}_+}, (\tau_n)_T)= (\tau_n)_T \circ\iota_{\id_{\underline{n}_+},T}$, which is precisely $\id_T$. To check that $F$ respects composition, consider composable morphisms 
    \[ (\underline{n}_+,T) \xrightarrow{(\varphi,f)} (\underline{m}_+, S) \xrightarrow{(\gamma,g)} (\underline{\ell}_+, R)\] 
    in $\tree$. By definition, $F(\gamma, g)\circ F(\varphi,f)$ and $F\left((\gamma, g)\circ (\varphi,f)\right)$ are given by the top and bottom composites, respectively, in the diagram
    \[\begin{tikzcd}[row sep=tiny]
        & \varphi^*(T)\rar["f"]\ar[ddr,"\iota_{\gamma,\varphi^*(T)}"] & S\ar[dr,"\iota_{\gamma,S}"] & &\\
        T\ar[ur,"\iota_{\varphi,T}"]\ar[dr,"\iota_{\varphi\gamma,T}"'] & & & \gamma^*(S)\rar["g"] & R.\\
        & (\varphi\gamma)^*(T)\rar["\tau_T"'] & \gamma^*\varphi^*(T)\ar[ur,"\gamma^*(f)"']\\
    \end{tikzcd}\] 

    Thus it suffices to prove that the two quadrilaterals commute. The one on the left commutes because both composites are given by the inclusion of $T$ as a subtree of $\gamma^*\varphi^*(T)$. The one on the right commutes as the top composite applies the morphism $f$ to $\varphi^*(T)$ and then attaches the corollas corresponding to $\gamma$, while the composite of the bottom sequence attaches the corollas for $\gamma$, and then acts as $f$ on the original tree and as the identity on the corollas attached. 

    We now want to show that $F$ is an equivalence of categories.  Given any tree $T$ in $\Omega$ with $n$ leaves, any choice of labeling of the leaves of $T$ produces an element $(\underline{n}_+, T)$ of $\tree$ such that $F(\underline{n}_+,T)=T$, so $F$ is surjective on objects.  So, it remains to show that $F$ is fully faithful.
    
    To show that $F$ is full, consider objects $(\underline{n}_+, T)$ and $(\underline{m}_+,S)$ of $\tree$; recall that the trees $T$ and $S$ are equipped with labelings $\chi_T\colon \underline{n} \xrightarrow{\cong} \ell(T)$ and $\chi_S\colon \underline{m}\xrightarrow{\cong} \ell(S)$.  Further recall from \cref{istrue} that any morphism $f\colon T\to S$ in $\Omega$ can be decomposed as $f=\delta_o\circ \delta_i\circ \zeta\circ \sigma$, 
    where $\sigma$ is a composite of degeneracy maps, $\zeta$ is an isomorphism, $\delta_i$ is a composite of inner face maps, and $\delta_o$ is a composite of outer face maps.  Hence, it suffices to show fullness for each of these types of generators individually, as the result then follows from the functoriality and surjectivity on objects of $F$. 
    
    Since degeneracies, isomorphisms, and inner face maps all preserve the leaves and root, we can give a single argument for these maps, leaving outer face maps as a separate case. So, suppose $f\colon T\to S$ is a map that induces a bijection $f\colon \ell(T)\to \ell(S)$ on the leaf sets and preserves the root. In particular, $T$ and $S$ have the same number of leaves, so $n=m$. If we give $T$ the labeling given by the composite $\underline{n}\xrightarrow{\chi_S} \ell(S) \xrightarrow{f^{-1}} \ell(T)$, then $f\colon (T, f^{-1}\chi_S)\to (S, \chi_S)$ is a map of labeled trees whose underlying map of trees is given by $f$. Let $\phi\colon \underline{n}_+\to \underline{n}_+$ be the pointed map given on non-basepoint elements by the composite $\underline{n}\xrightarrow{\chi_S} \ell(S) \xrightarrow{f^{-1}} \ell(T)\xrightarrow{\chi_T^{-1}} \underline{n}$. By construction, $\phi^*(T)$ is precisely $\id^*_{\underline{n}_+}(T, f^{-1}\chi_S)$, \emph{i.e.}, the $\underline{n}$-labeled tree $(T, f^{-1}\chi_S)$ with $1$-corollas attached to the leaves and the root, as in \cref{fig:idnotfunctorial}. Hence the component of $\tau_n$ on $(T, f^{-1}\chi_S)$ gives a map of labeled trees $(\tau_n)_T\colon \phi^*(T)\to (T, f^{-1}\chi_S)$. We then obtain a map of labeled trees $g\colon \phi^*(T) \to (T, f^{-1} \chi_S)\xrightarrow{f} (S,\chi_S)$, so $(\phi, g)\colon (\underline{n}_+, T)\to (\underline{n}_+, S)$ is a morphism in the Grothendieck construction. Moreover, $F(\phi, g) = f$ since the map of (non-labeled) trees 
    \[
    T \xrightarrow{\iota_{\phi, T}} \phi^*(T) \xrightarrow{(\tau_n)_T} T
    \] 
    is the identity. 

    Finally, suppose that $f\colon T\to S$ is an outer face map, so there exists an external vertex $v$ of $S$ such that $T$ is obtained from $S$ by pruning the outer edges connected to $v$ or, in the case where $S$ is a corolla, all but one of the outer edges. Let $e$ be the edge in $T$, and hence also in $S$, that was connected to $v$ and was not removed; observe $e$ is an outer edge of $T$.
    
    We now define a function $\varphi \colon \underline{m}_+ \to\underline{n}_+$ determined by this pruning and by the labelings of $T$ and $S$. The idea behind $\varphi$ is that, up to labeling, it records which outer edges of $S$ are pruned, as well as the location of that pruning in $T$. Concretely, the map $\varphi$ is defined as 
    \[\varphi(k)=\begin{cases}
        \chi^{-1}_T(\chi_S(k)) & \text{ if } \chi_S(k) \text{ was not removed in the pruning, } \\
        \chi^{-1}_T(e) & \text{ otherwise}.
    \end{cases}\]  
    Note that if $\chi_S(k)$ was not removed, it is an outer edge of $T$ as well, and hence $\chi_T^{-1}(\chi_S(k))$ makes sense. This formula is only ambiguous when $T$ is the stick tree and $S$ is a corolla, in which case $e$ is meant to denote the loose end that was attached to $v$. 
 
    This function is pointed. Indeed, if the root $\chi_S(+)$ of $S$ was pruned, then $e$ must be the root of $T$, and hence $\varphi(+)=\chi_T^{-1}(e)=+$. Otherwise, if the root $\chi_S(+)$ was not pruned, we have $\varphi(+)=\chi_T^{-1}(\chi_S(+))=+$ as $T$ and $S$ share the same root. 
 
    Now observe that $\varphi^*(T)$ is obtained from $T$ by attaching a corolla to the outer edge $e$ and a 1-corolla to every other outer edge. Let $g\colon \varphi^*(T)\to S$ be the map in $\T(\underline{m})$ given by the composite of the degeneracies that contracts the output edge of each of these unary vertices. Then $(\varphi,g) \colon (\underline{n}_+,T) \to (\underline{m}_+, S)$ is a morphism in $\tree$ such that $F(\varphi,g)=f$. 

    It remains to show that $F$ is faithful.  Suppose  $(\varphi, f),(\psi,g)\colon (\underline{n}_+,T) \to (\underline{m}_+,S)$ are morphisms such that $F(\varphi, f)=F(\psi,g)$; in other words, that the square
    \[\begin{tikzcd}
        T\rar\dar & \varphi^*(T)\dar["f"]\\
        \psi^*(T)\rar["g"'] & S
    \end{tikzcd}\] 
    commutes.  We need to show that $\varphi=\psi$ and $f=g$.  

    First, we want to show that $\varphi=\psi$ as functions $\underline{m}_+\to \underline{n}_+$.  We claim that we can uniquely reconstruct $\varphi$ from the data of $F(\varphi, f)$ together with the labelings $\chi_T$ and $\chi_S$; then the equality of $\varphi$ and $\psi$ follows from the assumption that $F(\varphi, f) = F(\psi, g)$. 
    
    For each $i\in \underline{n}$, consider the leaf $\chi_T(i)$ of $T$ as an edge of $\varphi^*(T)$. Then $f(\chi_T(i))$ is an edge of $S$ that is determined by $F(\varphi,f)$ and the labelings. Let $M_i\subseteq \ell(S)$ be the subset of leaves $\ell$ of $S$ such that $\ell\leq f(\chi_T(i))$. Note that if $i\neq i'$ then the leaves $\chi_T(i)$ and $\chi_T(i')$ are incomparable in $T$ and hence $f(\chi_T(i))$ and $f(\chi_T(i'))$ are incomparable in $S$, since morphisms of trees preserve incomparability of edges. It follows that $M_i\cap M_{i'} = \varnothing$. Therefore, for every $j\in \underline{m}$, the leaf $\chi_S(j)$ of $S$ belongs to at most one $M_i$. Moreover, since $f$ is a bijection on leaves, every leaf in $S$ must either belong to some $M_i$, or be in the image under $f$ of the corolla that was attached to the root of $T$.  An inspection of the composite $T \to \varphi^*(T) \to S$ shows that $\varphi$ must be defined by $\varphi(j) = +$ if $j\not\in M_i$ for any $i\in \underline{n}$, and otherwise $\varphi(j) = i$ for the unique $i\in \underline{n}$ with $j\in M_i$.  In the case of the stick tree this argument proceeds exactly the same, except that each instance of ``leaf'' should be replaced with ``loose end''. 
  
    All that remains is to check that $f=g$, for which it suffices to check that these maps induce the same map on both inner edges and on leaves.  Since all the inner edges of $\varphi^*(T) = \psi^*(T)$ are in the image of the map $\iota_T\colon T\to \varphi^*(T)$ and $f\circ \iota_T = g\circ \iota_T$, we see that $f$ and $g$ agree on inner edges.  These maps also agree on all leaves by the definition of morphisms in $\mathcal{T}(\underline{m})$, hence $f$ and $g$ agree on all edges and must be the same function.
\end{proof}

\section{Categories of trees with $G$-action}\label{sec:trees with G action}

In this section, we consider categories of trees equipped with an action of a finite group $G$.  In analogy with our treatment of non-equivariant trees, we first consider the category $\Omega^G$ of $G$-objects in the category $\Omega$, and then consider categories of trees with $G$-action and specified $G$-set of labels for the leaves.  As in the non-equivariant case, we obtain the former category as a Grothendieck construction defined in terms of the latter categories.

\subsection{The category $\Omega^G$}\label{sec:Omega upper G}

If $G$ is a finite group, let $BG$ denote the one-object category whose morphisms are given by the group $G$.  Recall that if $\mathcal C$ is a category and $G$ is a finite group, the category $\mathcal C^G = \Fun(BG, \mathcal C)$ is the category of \emph{$G$-objects in $\mathcal C$}. A functor $X\colon BG\to \mathcal C$ encodes the data of an object $X(\ast)$ that has an action by $g\in G$ given by a morphism in $\mathcal C$ of the form $X(g)\colon X(\ast)\to X(\ast)$. 

\begin{definition}
    The category of \emph{trees with $G$-action} is $\Omega^G = \Fun(BG, \Omega)$. 
\end{definition}

Explicitly, an object of $\Omega^G$ is a tree equipped with a $G$-action through root-preserving automorphisms that endow the sets of leaves, inner edges, and vertices with a $G$-action. There is a functor $\Omega^G\to \Omega$ that forgets the $G$-actions. An object of $\Omega^G$ is called a \emph{$G$-corolla} if its underlying tree in $\Omega$ is a corolla.  The following lemma gives a helpful characterization of when a morphism in $\Omega$ lifts to a morphism in $\Omega^G$.

\begin{lemma}\label{lem:equivariant maps in Omega}
    If $T$ and $S$ are objects in $\Omega^G$, then a morphism $f\colon T\to S$ in $\Omega$ is in $\Omega^G$ if and only if $f(g\cdot e) = g\cdot f(e)$ for all edges $e$ of $T$ and all $g\in G$.
\end{lemma} 

\begin{proof}
    We need to check that for all $g\in G$, the diagram
    \[
    \begin{tikzcd}
        T \ar[r,"f"] \ar[d, swap, "g_T"] & S \ar[d,"g_S"] \\
        T \ar[r,swap, "f"] & S
    \end{tikzcd}
    \]
    commutes in $\Omega$.  Here, $g_T$ and $g_S$ denote the morphisms in $\Omega$ that witness the action of $g\in G$ on $T$ and $S$, respectively. As explained in \cref{rmk:poset}, morphisms in $\Omega$ are completely determined by what they do on edges, so it suffices to check that $f\circ g_T = g_S\circ f$ as functions on the edge sets of $T$ and $S$. That is, we need to know that $f(g\cdot e) = g\cdot f(e)$ for all edges $e$ in $T$ and all $g\in G$, which is true by assumption.
\end{proof} 

As in $\Omega$, the morphisms of $\Omega^G$ are generated by certain face and degeneracy maps. 

\begin{definition}
    A map $T\to S$ in $\Omega^G$ is a \emph{degeneracy} if there is an edge $e$ of $S$ so that $T$ is obtained from $S$ by splitting each of the edges $g\cdot e$ of $S$ for all $g\in G$ by adding a unary vertex, and the map witnesses the merger of these edges. We write $\sigma_{[e]}$ for this degeneracy.
\end{definition}

That is, a degeneracy removes an  orbit of unary vertices by merging the corresponding edges. Note that there is a factorization
\[\sigma_{[e]} =  \sigma_{g_1e}\circ \sigma_{g_2e} \circ\dots\circ\sigma_{g_ne}\] 
in $\Omega$, ranging over all elements $g_i e$ in the orbit of $e$, where $\sigma_{e}$ is the (non-equivariant) degeneracy map for the edge $g\cdot e$ of $S$. We may write this composite unambiguously, as the degeneracy maps $\sigma_{g\cdot e}$ and $\sigma_{g'\cdot e}$ commute for $g,g'$ such that $g\cdot e \neq g'\cdot e$. 

\begin{example}
    For $G=\mathbb{Z}/2$, \cref{fig:Gtree degen} gives an example of a degeneracy map, which we can think of as the opposite of splitting of edge the target by the inclusion of an orbit of unary vertices.  The $G$-fixed part of the tree is colored in black and the other $G/e$ orbits are indicated by different colors, with edge-labelings $\pm$ to indicate the different elements of the orbit.   
    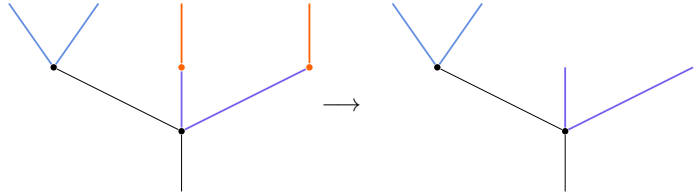
\begin{figure}
        \centering
    \[
    \scalebox{0.85}{
    \begin{tikzpicture} 
    [level distance=10mm, 
    every node/.style={fill, circle, minimum size=.1cm, inner sep=0pt}, 
    level 1/.style={sibling distance=20mm}, 
    level 2/.style={sibling distance=20mm}, 
    level 3/.style={sibling distance=14mm},
    level 4/.style={sibling distance=7mm}]
    \node (tree) at (-3,0) [style={color=white}] {} [grow'=up] 
    child {node (x) (level2) {} 
	child{ node (level3) {}
		child[draw=myblue, thick]
		child[draw=myblue, thick]
	}
	child[draw=mypurple, thick]{ node[color=mypurple] (level3) {}
            child[draw=myorange, thick]
        }
    child[draw=mypurple, thick]{ node[color=mypurple] (level3) {}
            child[draw=myorange, thick]
        }
    };
    \node[fill=white, label=$\longrightarrow$] at (-0.5,1) {};

    \node (tree) at (3,0) [style={color=white}] {} [grow'=up] 
    child {node (level2) {} 
	child{ node (level3) {}
		child[draw=myblue, thick]
		child[draw=myblue, thick]
	}
	child[draw=mypurple, thick]
        child[draw=mypurple, thick]
    };

    \tikzstyle{every node}=[]
    \node at ($(x) + (-15mm, 14mm)$){\textcolor{myblue}{$-$}};
    \node at ($(x) + (-25mm, 14mm)$){\textcolor{myblue}{$+$}};
    \node at ($(x) + (-4mm, 14mm)$){\textcolor{myorange}{$+$}};
    \node at ($(x) + (24mm, 14mm)$){\textcolor{myorange}{$-$}};
    \node at ($(x) + (-4mm, 6mm)$){\textcolor{mypurple}{$+$}};
    \node at ($(x) + (18mm, 6mm)$){\textcolor{mypurple}{$-$}};

    \node at ($(y) + (-4mm, -14mm)$){\textcolor{mypurple}{$+$}};
    \node at ($(y) + (18mm, -14mm)$){\textcolor{mypurple}{$-$}};
    \node at ($(y) + (-15mm, -6mm)$){\textcolor{myblue}{$-$}};
    \node at ($(y) + (-25mm, -6mm)$){\textcolor{myblue}{$+$}};
    \end{tikzpicture}
    }
    \]
    \caption{A degeneracy map between trees with $G$-action}
        \label{fig:Gtree degen}
    \end{figure}
\end{example}

\begin{definition}
    A map $T\to S$ in $\Omega^G$ is an \emph{inner face map} if there is an edge $e$ in $S$ such that $T$ is obtained from $S$ by contracting the edges $g\cdot e$ of $S$ for all $g\in G$. We write $\delta_{[e]}$ for this inner face map.
\end{definition}

That is, an inner face map is the opposite of a contraction of an orbit of inner edges. As before, there is a factorization of $\delta_{[e]}$ as the composite of non-equivariant inner face maps $\delta_{ge}$ for all $g\cdot e$ in the orbit of $e$.

\begin{example}
Adopting the same coloring conventions as the previous example, an example of an inner face map is given by
    \[
    \scalebox{0.85}
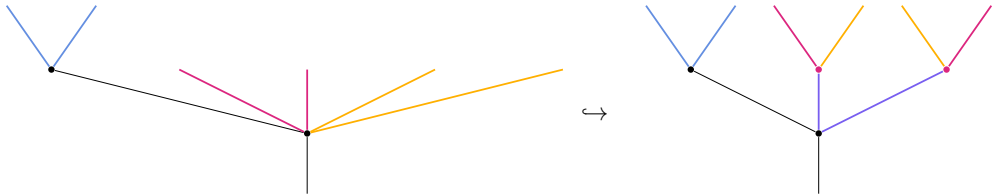
{
    \begin{tikzpicture} 
    [level distance=10mm, 
    every node/.style={fill, circle, minimum size=.1cm, inner sep=0pt}, 
    level 1/.style={sibling distance=20mm}, 
    level 2/.style={sibling distance=20mm}, 
    level 3/.style={sibling distance=14mm},
    level 4/.style={sibling distance=7mm}]
    \node (tree) at (-5,0) [style={color=white}] {} [grow'=up] 
    child {node (level2) {} 
	child{ node (level3) {}
		child[draw=myblue, thick]
		child[draw=myblue, thick]
	}
	child[draw=myred, thick]
	child[draw=myred, thick]
    child[draw=myyellow, thick]
	child[draw=myyellow, thick]
    };
    \node[fill=white, label=$\hookrightarrow$] at (-0.5,1) {};
    \node (tree) at (3,0) [style={color=white}] {} [grow'=up] 
    child {node (y) (level1) {} 
	child{ node {}
		child[draw=myblue, thick]
		child[draw=myblue, thick]
	}
        child[draw=mypurple, thick]{ node[color=mypurple] {}
    		child[draw=myred, thick]
    		child[draw=myyellow, thick]
    	}
        child[draw=mypurple, thick]{ node [color=mypurple] {}
    		child[draw=myyellow, thick]
    		child[draw=myred, thick]
    	}
    };
    \tikzstyle{every node}=[]
    \node at ($(x) + (-65mm, 14mm)$){\textcolor{myblue}{$+$}};
    \node at ($(x) + (-55mm, 14mm)$){\textcolor{myblue}{$-$}};
    \node at ($(x) + (-40mm, 8mm)$){\textcolor{myred}{$+$}};
    \node at ($(x) + (-22mm, 8mm)$){\textcolor{myred}{$-$}};
    \node at ($(x) + (-12mm, 6mm)$){\textcolor{myyellow}{$+$}};
    \node at ($(x) + (14mm, 6mm)$){\textcolor{myyellow}{$-$}};

    \node at ($(y) + (-25mm, -6mm)$){\textcolor{myblue}{$+$}};
    \node at ($(y) + (-15mm, -6mm)$){\textcolor{myblue}{$-$}};
    \node at ($(y) + (-5mm, -6mm)$){\textcolor{myred}{$+$}};
    \node at ($(y) + (+25mm, -6mm)$){\textcolor{myred}{$-$}};
    \node at ($(y) + (5mm, -6mm)$){\textcolor{myyellow}{$+$}};
    \node at ($(y) + (15mm, -6mm)$){\textcolor{myyellow}{$-$}};
    \node at ($(y) + (-2mm, -14mm)$){\textcolor{mypurple}{$+$}};
    \node at ($(y) + (15mm, -14mm)$){\textcolor{mypurple}{$-$}};
    \end{tikzpicture}
    }
    \]
which we think of as the opposite of a contraction of the $G/e$-orbit of purple edges.
\end{example}

\begin{definition}
    Assume $S$ is not a $G$-corolla.  Let $v$ an external vertex of $S$, and let $T$ be the $G$-tree obtained by removing the orbit of $v$ and all outer edges attached to the orbit of $v$.  That is, $V(T)=V(S)\setminus\{G\cdot v\}$ and $E(T)$ is the subset of $E(S)$ consisting of all edges except the outer edges connected to $G\cdot v$. The map $\delta_{[v]}\colon T \to S$ induced by the inclusion of edges $E(T)\to E(S)$ is called an \emph{outer face map}. 

    When $S$ is a $G$-corolla and $e$ is a $G$-fixed leaf of $S$, then there is another type of outer face  $\delta_{[e]}\colon \eta \to S$ given by including the stick tree $\eta$ as the edge $e$; note that there are $|\ell(S)^G|$ such maps, corresponding to each fixed leaf. This case is the only one in which the external vertex $v$ does not determine the pruning uniquely.  
\end{definition}

The first kind of outer face map grafts the same corolla onto each leaf in an orbit. We emphasize that the same underlying corolla must be grafted onto each leaf in the orbit, to ensure that the resulting tree has a well-defined $G$-action. Along the same lines, if the $G$-corolla is grafted onto the root, as in \cref{ex:equivar outer face}, then the grafting must merge the root with a $G$-fixed leaf of the $G$-corolla. As before, an outer face map $\delta_{[v]}$ in $\Omega^G$ may be written as a composite of non-equivariant outer face maps $\delta_{gv}$ in $\Omega$, ranging across all $gv$ in the orbit of $v$.

\begin{example}\label{ex:equivar outer face}
    Adopting the same conventions as in the previous examples, an example of an outer face map is given by
    \[
    \scalebox{0.85}
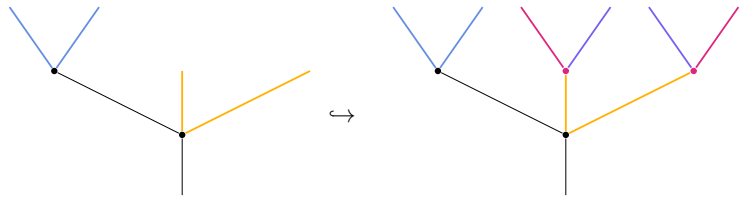
{
    \begin{tikzpicture} 
    [level distance=10mm, 
    every node/.style={fill, circle, minimum size=.1cm, inner sep=0pt}, 
    level 1/.style={sibling distance=20mm}, 
    level 2/.style={sibling distance=20mm}, 
    level 3/.style={sibling distance=14mm},
    level 4/.style={sibling distance=7mm}]
    \node (tree) at (-3,0) [style={color=white}] {} [grow'=up] 
    child {node (x) (level2) {} 
	child{ node (level3) {}
		child[draw=myblue, thick]
		child[draw=myblue, thick]
	}
	child[draw=myyellow, thick]
	child[draw=myyellow, thick]
    };
    \node[fill=white, label=$\hookrightarrow$] at (-0.5,1) {};
    \node (tree) at (3,0) [style={color=white}] {} [grow'=up] 
    child {node (y) (level1) {} 
	child{ node [vertex] {}
		child[draw=myblue, thick] {}
		child[draw=myblue, thick]
	}
        child[draw=myyellow, thick]{ node[color=myyellow] {}
    		child[draw=myred, thick]
    		child[draw=mypurple, thick]
    	}
        child[draw=myyellow, thick]{ node [color=myyellow] {}
    		child[draw=mypurple, thick]
    		child[draw=myred, thick]
    	}
    };

    \tikzstyle{every node}=[]
    \node at ($(x) + (-25mm, 14mm)$){\textcolor{myblue}{$+$}};
    \node at ($(x) + (-15mm, 14mm)$){\textcolor{myblue}{$-$}};
    \node at ($(x) + (-2mm, 6mm)$){\textcolor{myyellow}{$+$}};
    \node at ($(x) + (16mm, 6mm)$){\textcolor{myyellow}{$-$}};

    \node at ($(y) + (-25mm, -6mm)$){\textcolor{myblue}{$+$}};
    \node at ($(y) + (-15mm, -6mm)$){\textcolor{myblue}{$-$}};
    \node at ($(y) + (-2mm, -14mm)$){\textcolor{myyellow}{$+$}};
    \node at ($(y) + (16mm, -14mm)$){\textcolor{myyellow}{$-$}};
    \node at ($(y) + (-5mm, -6mm)$){\textcolor{myred}{$+$}};
    \node at ($(y) + (25mm, -6mm)$){\textcolor{myred}{$-$}};
    \node at ($(y) + (5mm, -6mm)$){\textcolor{mypurple}{$-$}};
    \node at ($(y) + (15mm, -6mm)$){\textcolor{mypurple}{$+$}};
    \end{tikzpicture}
    }
    \]
    which we can think of as either an inclusion of a subtree or as the opposite of the removal of the orbit of red vertices and their incoming leaves. Another kind of outer face map grafts the root onto a corolla
    \[
    \scalebox{0.85}
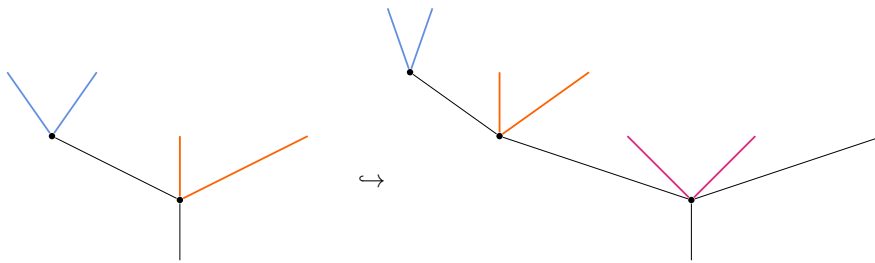
{
    \begin{tikzpicture} 
    [level distance=10mm, 
    every node/.style={fill, circle, minimum size=.1cm, inner sep=0pt}, 
    level 1/.style={sibling distance=20mm}, 
    level 2/.style={sibling distance=20mm}, 
    level 3/.style={sibling distance=14mm},
    level 4/.style={sibling distance=7mm}]
    \node (tree) at (-3,0) [style={color=white}] {} [grow'=up] 
    child {node (x) (level2) {} 
	child{ node (level3) {}
		child[draw = myblue, thick]
		child[draw = myblue, thick]
	}
	child[draw = myorange, thick]
	child[draw = myorange, thick]
    };
    \node[fill=white, label=$\hookrightarrow$] at (0,1) {};
    \node (tree) at (5,0) [style={color=white}] {} [grow'=up] 
    child {node[color=myred] (y) (level1) {} 
	child{ node {}
		child{ node {}
            child[draw = myblue, thick]
            child[draw = myblue, thick]}
		child[draw = myorange, thick]
        child[draw = myorange, thick]
	}
        child[draw = myred, thick]
        child[draw = myred, thick]
    };

    \tikzstyle{every node}=[]
    \node at ($(x) + (-25mm, 14mm)$){\textcolor{myblue}{$+$}};
    \node at ($(x) + (-15mm, 14mm)$){\textcolor{myblue}{$-$}};
    \node at ($(x) + (-2mm, 6mm)$){\textcolor{myorange}{$+$}};
    \node at ($(x) + (16mm, 6mm)$){\textcolor{myorange}{$-$}};

    \node at ($(y) + (-18mm, 4mm)$){\textcolor{myblue}{$+$}};
    \node at ($(y) + (-10mm, 4mm)$){\textcolor{myblue}{$-$}};
    \node at ($(y) + (-2mm, -4mm)$){\textcolor{myorange}{$+$}};
    \node at ($(y) + (12mm, -4mm)$){\textcolor{myorange}{$-$}};
    \node at ($(y) + (18mm, -14mm)$){\textcolor{myred}{$-$}};
    \node at ($(y) + (36mm, -14mm)$){\textcolor{myred}{$+$}};
    \end{tikzpicture}
    }
    \]
    which we can think of either as an inclusion or as the opposite of the removal of the red vertex and incoming red leaves. We emphasize that the corolla onto which we have grafted has a leaf $G$-set isomorphic to $G/G\amalg G/e\amalg G/G$.
\end{example}

We now show that these equivariant faces and degeneracies generate the morphisms in $\Omega^G$. This proof closely follows the one for the non-equivariant version of this result, which can be found in \cite[Proposition 3.9, Proposition 3.10(2)]{heutsmoerdijk}.

\begin{proposition}\label{istrueG}
    Every morphism $f\colon T\to S$ in $\Omega^G$ can be factored as a composite $f = \delta_o \circ \delta_i \circ\alpha\circ\sigma$ of morphisms in $\Omega^G$, where $\sigma$ is a composite of degeneracy maps, $\alpha$ is an isomorphism, $\delta_i$ is a composite of inner face maps, and $\delta_o$ is a composite of outer face maps. 
\end{proposition}

\begin{proof} 
    First, we consider the action of $f$ on underlying sets of edges, and factor this map as a surjection followed by an injection 
    \[E(T)\to (E(T)/\sim)\to E(S),\] 
    where $\sim$ is the equivalence relation in which $e_1\sim e_2$ if and only if $f(e_1)=f(e_2)$. The fact that $f$ is an equivariant map implies that $E(T)/\sim$ is a $G$-set and both functions in this factorization are equivariant.  To start, we choose an equivalence class in $E(T)$. As explained in \cite{heutsmoerdijk}, such an equivalence class is given by a path of edges $e_0\geq e_1 \geq \dots \geq e_k$ connected by unary vertices $v_1, \dots , v_k$.  Equivariance of $f$ implies that the equivalence class of $g\cdot e_i$ consists precisely of the path $g\cdot e_0\geq g\cdot e_1 \geq \dots \geq g\cdot e_k$ connected by the unary vertices $g\cdot v_1, \dots , g\cdot v_k$. There is a composite of degeneracy maps $\sigma'\colon T\to T'$, where $T'$ is the tree obtained by merging the edges $g\cdot e_i$ for all $g\in G$ into a single edge, and then removing the unary vertices $g\cdot v_i$ for all $i$. The dendroidal identities guarantee that the order in which we perform each of these degeneracies is irrelevant, as the final composite is always the same. Applying this reasoning to the orbit of each equivalence class of edges, we obtain a composite of equivariant degeneracy maps $\sigma\colon T\to \hat{T}$ with $E(\hat{T}) \cong E(T)/\sim$. Once again, the dendroidal identities guarantee that the resulting map is independent of the order in which we choose the equivalence classes.

    Thus, $f$ factors as $i\circ \sigma$ for some equivariant map $i\colon T'\to S$ that is injective on edges. This map $i$ further factors as an isomorphism $\alpha\colon T'\to T''$ in $\Omega^G$ followed by a map $\iota\colon T''\to S$ that is an inclusion of edges. Writing $\ell_1,\dots,\ell_n$ for the leaves of the tree $T''$ and $r$ for its root, the map $\iota$ factors further as
    \[ T'' \xrightarrow{\zeta} S' \xrightarrow{\varphi} S, \] 
    where $S'$ is the maximal subtree of $S$ with leaves $\iota(\ell_1), \dots, \iota(\ell_n)$ and root $\iota(r)$, and $\zeta$ is the map $\iota$ with restricted target $S'$, and $\varphi$ is the inclusions of the subtree. At the level of edges all of these maps are $G$-equivariant inclusions.

    We claim that $\zeta$ is a composite of inner face maps. Recall that each vertex $v$ in $T''$ is mapped by the morphism $\zeta$ to a subtree of $S'$. If $\zeta$ maps every vertex in $T''$ to a corolla, then $\zeta$ must be the identity map, and we are done. Otherwise, there is a vertex $v$ in $T$ such that the subtree $\zeta(v)\subseteq S'$, as defined before \cref{defn: dendroidal cat}, has an inner edge $e$. Note that the fact that $\zeta$ is equivariant implies that every vertex $g\cdot v$ must also produce a subtree with an inner edge $g\cdot e$. We can then factor $\zeta\colon T''\to S'$ as the composite 
    \[T''\xrightarrow{\zeta'} \hat{T}\xrightarrow{\delta_{[e]}} S, \] 
    where $\delta_{[e]}$ is the corresponding equivariant inner face map. We can iterate this process on the inner edges in the subtree $\zeta(v)$, for each vertex $v$ in $T$, to obtain the desired decomposition of $\zeta$. 
     
    Lastly, we claim that $\varphi$ can be expressed as a composite of equivariant outer face maps. Informally, the idea is that since $S'$ is a $G$-subtree of $S$, we can reconstruct the rest of $S$ by attaching corollas to $S'$ one level at a time, first growing upward from the leaves of $S'$, and then growing downward from the root of $S'$.  To ensure that these attachments are done via equivariant outer faces, we must take care to attach these corollas in such a way that each step of the process respects the $G$-action. 
     
    More concretely, let $v$ be an external vertex of $S$ such that $v\not\in S'$. Since $S'$ is a $G$-subtree, it follows that $g\cdot v \not\in S'$ for all $g\in G$, so $S'$ lacks the corollas attached to $g\cdot v$ for all $g\in G$, but possibly also other parts of $S$. Then $\varphi$ factors as
    \[S' \xrightarrow{\varphi'} S''\xrightarrow{\delta_{[v]}} S,\]
    where $\varphi'$ is again the inclusion of a subtree.
    We can repeat this process by taking an external vertex of the new tree $S''$ that is not in $S'$. This process ends once we have removed all the vertices of $S$ not in $S'$, as each eventually becomes an external vertex once enough corollas have been removed.
\end{proof}

\subsection{Categories of trees with $G$-action and leaves labeled by a $G$-set}\label{sec:Omega upper G is Groth}

We now introduce categories of trees with $G$-action whose leaf labels are given by a specific $G$-set. 

\begin{definition} \label{def:T(A)} 
    Given a finite $G$-set $A$, an \emph{$A$-labeled tree with $G$-action} is an object in $\Omega^G$ equipped with an equivariant labeling bijection between $A$ and the $G$-set of leaves. A \emph{morphism} of $A$-labeled trees with $G$-action is a map in $\Omega^G$ that preserves the labels. Let $\T^G(A)$ denote the category of $A$-labeled trees with $G$-action and label-preserving morphisms between them. 
\end{definition}

Observe that, just as in the non-equivariant case, the morphisms of $\T^G(A)$ are generated by equivariant inner face and degeneracy maps and isomorphisms, since the other generating morphisms in $\Omega^G$ necessarily change the leaves.  

\begin{remark}\label{rmk:comparison with our trees}
    As in the non-equivariant case, note that $\T^G(A)$ is a slightly different category than its counterpart of the same name in \cite{BBCCM:trees}.  As in the non-equivariant case, the main difference lies in including the corollas and allowing unary vertices, enabling the inclusion of degeneracy maps.  There we also took the objects to be isomorphism classes of trees, whereas here we consider all trees.  
\end{remark}

\begin{example} \label{Gtree ex}
    Let $G = \{1,i,-1,-i\}$, and $A = \{x,ix,y,-y,iy,-iy\}$. The following is an example of two $A$-labeled trees with $G$-action and a morphism between them in the category $\mathcal{T}^G(A)$:
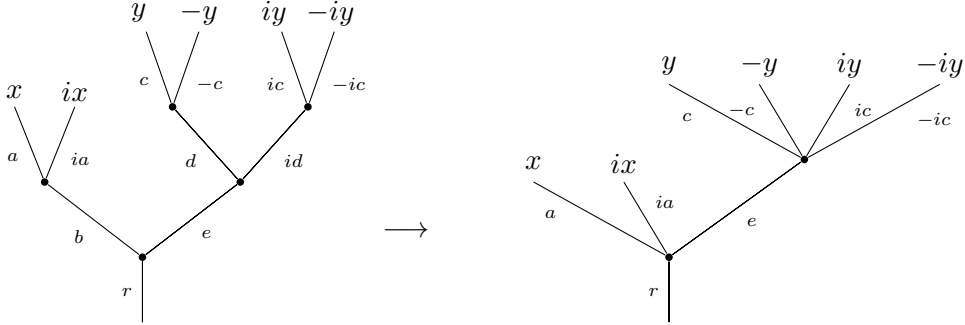
     \[ \begin{tikzpicture} 
    [level distance=10mm, grow'=up, auto,
    every node/.style={font=\tiny},
    level 2/.style={sibling distance=26mm}, 
    level 3/.style={sibling distance=18mm},
    level 4/.style={sibling distance=7mm}]
    \node (tree) at (7,0) {}
        child {node [vertex] {} 
	        child{ node (xy) [vertex] {}
		        child[sibling distance=8mm, draw=myblue] {edge from parent node {\textcolor{myblue}{$\pm 1$}}}
			    child[sibling distance=8mm, draw=myblue] {edge from parent node [swap] {\textcolor{myblue}{$\pm i$}}}
			    edge from parent node {}
	        }
    	    child{ node [vertex] {}
	    	    child[draw=mypurple] { node (xiy) [vertex] {}
		    	    child[draw=myred] {edge from parent node {\textcolor{myred}{$1$}}}
    			    child[draw=myred] {edge from parent node [swap] {\textcolor{myred}{$-1$}}}
    			    {edge from parent node {\textcolor{mypurple}{$\pm 1$}}}
        		}
	        	child[draw=mypurple] { node (yiy) [vertex] {}
		        	child[draw=myred] {edge from parent node {\textcolor{myred}{$i$}}}
			        child[draw=myred] {edge from parent node [swap] {\textcolor{myred}{$-i$}}}
			        {edge from parent node [swap]{\textcolor{mypurple}{$\pm i$}}}
        		}
        		{edge from parent node [swap] {}}
	        }
	        {edge from parent node {}}
        };
    
    \tikzstyle{level 2}=[{sibling distance=12mm}]
    \tikzstyle{level 3}=[{sibling distance=12mm}]   
    \node (tree2) [style={color=white}] {}
        child{node (r) [vertex] {}
            child[draw=myblue] {edge from parent node [near end] {\textcolor{myblue}{$\pm 1$}}}
            child[draw=myblue] {edge from parent node [swap] {\textcolor{myblue}{$\pm i$}}}
            child{edge from parent [draw=none]}
            child[level distance=13mm]{node (s) [vertex] {}
                child[level distance=10mm, draw=myred] {edge from parent node [near end] {\textcolor{myred}{$1$}}}
                child[level distance=10mm, draw=myred] {edge from parent node [very near end]{\textcolor{myred}{$-1$}}} 
                child[level distance=10mm, draw=myred] {edge from parent node [swap, very near end] {\textcolor{myred}{$i$}}}
                child[level distance=10mm, draw=myred] {edge from parent node [swap, near end ]{\textcolor{myred}{$-i$}}}
                {edge from parent node [swap] {}}
            }
            {edge from parent node {}}
        };
    
    \tikzstyle{every node}=[]
    \node[label=$\longrightarrow$] at (3.5,10mm) {};
    \node at ($(xy) + (4mm, 12mm)$){$ix$};
    \node at ($(xy) + (-4mm, 12mm)$){$x$};
    
    \node at ($(xiy) + (-4.5mm, 12.5mm)$){$y$};
    \node at ($(xiy) + (3.5mm, 12.5mm)$){$-y$};
    \node at ($(yiy) + (-4.5mm, 12.5mm)$){$iy$};
    \node at ($(yiy) + (3mm, 12.5mm)$){$-iy$};
    
    \node at ($(r) + (-18mm,12.5mm)$){$x$};
    \node at ($(r) + (-6mm,12.5mm)$){$ix$};
    \node at ($(s) + (-18mm,12.5mm)$){$y$};
    \node at ($(s) + (-6mm,12.5mm)$){$-y$};
    \node at ($(s) + (6mm,12.5mm)$){$iy$};
    \node at ($(s) + (18mm,12.5mm)$){$-iy$};
    \end{tikzpicture}. \]

    Note that while the picture above preserves the planar structures of the trees, this is not required.
\end{example}

In the remainder of this section, we prove an analogue of \cref{OmegaisTree} for $\Omega^G$. To describe the equivariant version of \cref{construction:oplaxT}, we also need to consider maps $T\to T'$ where $T$ is an $A$-labeled tree and $T'$ is a $B$-labeled tree for two non-isomorphic $G$-sets $A$ and $B$.  The idea for such morphisms is that they should correspond to the outer face maps in the dendroidal category $\Omega^G$. Thus, informally, they correspond to the procedure of grafting, but notationally, it is easier to describe these morphisms via the opposite procedure of pruning.  However, in the context of trees with $G$-action, this grafting or pruning needs to be applied uniformly to every vertex in the same $G$-orbit, just as we saw in the previous subsection.

Let $\Fin_{G,*}$ denote the category of finite pointed $G$-sets and whose maps are equivariant pointed functions. We may write an object of $\Fin_{G,*}$ as $A_+ = A\amalg G/G$ for some (unpointed) finite $G$-set $A$ and $+=G/G$ the basepoint.

\begin{construction}\label{constr:oplaxTG} 
    We construct an assignment $\T^G$ from $\Fin_{G,*}^{\op}$ to $\Cat$ as follows.
    \begin{itemize}
        \item On objects, $\T^G$ maps $A_+$ to the category $\T^G(A)$ of \cref{def:T(A)}. 
        
        \item On morphisms, given an equivariant function $\varphi\colon B_+\to A_+$ in $\Fin_{G,*}$, we define the functor $\T^G(\varphi) = \varphi^* \colon \T^G(A)\to \T^G(B)$ as follows.
        \begin{itemize}
            \item For an object $T$ in $\T^G(A)$, the tree $\varphi^*(T)$ is obtained from $T$ by replacing each leaf labeled by $a\in A_+$ with the corolla labeled by $\varphi^{-1}(a)$. This includes the case of $a=+$, where we attach the corolla labeled by $\varphi^{-1}(+)$ to the root, with the new root being the edge labeled by $+\in \varphi^{-1}(+)$.  We claim that $\varphi^*(T)$ is then an object in $\T^G(B)$:  the action of $G$ on $\varphi^*(T)$ extends the action of $G$ on $T$, and on leaves it is defined by the action of $G$ on $B$. 
            
            \item A morphism $f\colon T\to S$ in $\T^G(A)$ is  sent to the morphism $\varphi^*(f)\colon\varphi^*(T)\to\varphi^*(S)$ given by $\varphi^*(f)(a)=a$ for all leaves $a\in A$, and by $\varphi^*(f)(e)=f(e)$ for all edges $e$ previously present in $T$. Note that $\varphi^*f$ is equivariant, since $f$ and $\varphi$ are equivariant by assumption.
        \end{itemize}
    \end{itemize}
\end{construction}

As in the non-equivariant setting, this assignment is not functorial but rather defines an oplax functor.

\begin{proposition}\label{TG is oplax}
    The assignment $\mathcal{T}^G$ defines an oplax functor.
\end{proposition}

\begin{proof}
    To show that $\mathcal{T}^G$ defines an oplax functor, we must provide the data of natural transformations 
    \[\tau_A\colon (\id_{A_+})^* \Rightarrow \id_{\T(A)} \quad \text{and} \quad \tau_{\gamma, \varphi} \colon(\varphi \circ \gamma)^* \Rightarrow \gamma^* \circ \varphi^*\]
    for all objects $A_+$ of $\Fin_{G, *}^{\op}$ and all maps $A_+ \xrightarrow{\gamma} B_+ \xrightarrow{\varphi} C_+$ in $\Fin_{G, *}^{\op}$. Consider the underlying pointed set $|A|_+$ of $\Fin_*$ and the corresponding underlying pointed functions $|\gamma|, |\varphi|$; we claim that we can define $\tau_A=\tau_{|A|}$ and $\tau_{\gamma,\varphi} = \tau_{|\gamma|, |\varphi|}$. where the latter are the structure maps of the oplax functor $\T \colon \Fin_* \to \Cat$ of \cref{prop:functorT}. 

    The fact that these correspondences are natural follows directly from the naturality of $\tau_{|A|}$ and $\tau_{|\gamma|,|\varphi|}$, and the same is true for the required unitality and associativity conditions. It only remains to check that for every object $T$ of $\mathcal{T}^G (A)$, each of the components $(\tau_A)_T$ and $(\tau_{\gamma,\varphi})_T$ of the natural transformations are $G$-equivariant. This verification is straightforward, using \cref{lem:equivariant maps in Omega} and inspecting the descriptions of these maps given in \cref{prop:functorT}.
\end{proof}

Let $\treeg$ denote the Grothendieck construction of the oplax functor $\T^G$. The main result of this section is the following.

\begin{theorem}\label{GOmega is GTree}
   There is an equivalence of categories $F \colon \treeg \xrightarrow{\simeq} \Omega^G$. 
\end{theorem}

\begin{proof}
    The argument is very similar to the one given for \cref{OmegaisTree}, so we omit many of the details that follow analogously. Define a functor 
    \[ F \colon \treeg \to\Omega^G \] 
    on objects by $F(A_+,T)=T$, observing that, in contrast to the non-equivariant case, $A_+$ has a $G$-action.  Recall that a morphism $(\varphi, f)\colon (A_+,T)\to(B_+,S)$ in $\treeg$ is given by a $G$-equivariant function $\varphi\colon B\to A$ together with a map $f\colon \varphi^*(T)\to S$ in $\T(B)$, where $\varphi^*(T)$ is obtained from $T$ as described in \cref{constr:oplaxTG}.  Consider the unique composite of outer face maps $\iota_{\varphi,T}\colon T\to\varphi^*(T)$ in $\Omega^G$ that attaches the necessary corollas.
    This map, which is the inclusion of $T$ into $\varphi^*(T)$, is equivariant on edges by definition, and hence defines a morphism in $\Omega^G$ by \cref{lem:equivariant maps in Omega}.
    We then define $F(\varphi, f)$ as the composite
    \[T \xrightarrow{\iota_{\varphi,T}} \varphi^* (T)\xrightarrow{f} S\]
    and observe that it agrees with the functor from \cref{OmegaisTree} after forgetting the $G$-actions.  It follows directly that $F$ is a functor. 
    
    The argument that $F$ is an equivalence of categories is very similar to the non-equivariant case; in particular, $F$ is surjective on objects by the same argument as before.
 
    To prove $F$ is full, fix objects $(A,T)$ and $(B,S)$ in $\treeg$ and a morphism $f \colon T \rightarrow S$ in $\Omega^G$.  As in the non-equivariant case, we decompose any such $f$ as the composite of degeneracy maps, isomorphisms, inner face maps, and outer face maps, and establish fullness for each of these four kinds of maps separately.  The proofs for degeneracy maps, inner face maps, and isomorphisms are completely analogous to the arguments in the non-equivariant setting.

    In the case when $f\colon T\to S$ is an outer face map, there is a bit more subtlety in the equivariant case, since we need to consider orbits. As in the non-equivariant case, we can think of the root of $T$ as another leaf, referring to it by the term ``leaf", but noting that it is still designated as the root by its label $+$. The map $f$ is equal to $\delta_{[v]}$, where $v$ is an external vertex of $S$. Let $a\in A_+$ be the edge connected to $v$, considered as a leaf of $T$. Then the leaves of $T$ in the orbit of $\chi_T(a)$ are sent to edges of $S$;  any leaf $\chi_T(a')$ that is not in the orbit of $\chi_T(a)$ is sent to a leaf $f(\chi_T(a'))$ of $S$, corresponding to the element $\chi^{-1}_S(f(\chi_T(a'))$ of $B$.  Let $B'\subseteq B_+$ be the $G$-subset of labels of the form $\chi^{-1}_S(f(\chi_T(a'))$ for some $a'\not\in G\cdot a$.

    We then define a pointed equivariant function $\varphi\colon B_+\to A_+$ as follows. On $B'$, we define
    \[\varphi(b)=a',   
    \] 
    where $a'$ is the unique element of $A_+$ such that $b=\chi^{-1}_S(f(\chi_T(a'))$. 
    To define $\varphi$ on $B_+\setminus B'$, it suffices to specify the map on the labels of the leaves of $S$ that belong to the corolla attached to the inner edge $f(\chi_T(a))$ and extend equivariantly; we do so by sending all these labels to $a$.
    
    Then $\varphi^*(T)$ is obtained from $T$ by attaching a corolla to each leaf in the orbit of $\chi_T(a)$ and a 1-corolla to every other leaf.  Let $g\colon \varphi^*(T)\to S$ denote the map in $\T(B)$ given by the composite of the degeneracies that contracts the output edge of each of these unary vertices and note that $g$ is equivariant. Thus we have defined a morphism $(\varphi,g)\colon (A_+,T)\to(B_+, S)$ in $\int \T^G$ such that $F(\varphi,g)=f$. 

    A similar argument works for the other type of outer face map, namely, that which attaches a corolla to a stick. In that case, the role of $a$ is played by the label of the loose end of the stick where the corolla is attached.

    It remains to show that $F$ is faithful, which may be argued in exactly the same way as in the proof of \cref{OmegaisTree}.   
\end{proof}

\section{Categories of genuine equivariant trees} \label{sec:Gtree}  

In this section, we present fully equivariant versions of the different categories of trees presented in the previous sections.  Preliminary to introducing group actions, we review the theory of forests.  Then we proceed to define Bonventre and Pereira's equivariant dendroidal category $\Omega_G$ associated to a finite group $G$. The idea is that an object of $\Omega_G$ represents a tree with an $H$-action for some $H\leq G$, and morphisms in $\Omega_G$ are generated by morphisms in $\Omega^H$ together with morphisms that change the subgroup $H\leq G$. This idea is made precise in \cref{prop:OmegaGH Grothendieck}, where we exhibit $\Omega_G$ as a Grothendieck construction. Rather than using trees with an action of $H$, it is convenient to instead use a category whose objects are forests with $G$-action where the $G$-set of roots is isomorphic to $G/H$.  Finally, in \cref{theresult}, we show that $\Omega_G$ can also be modeled as an iterated Grothendieck construction on categories of equivariant labeled trees, combining \cref{prop:OmegaGH Grothendieck} with \cref{GOmega is GTree}.

\subsection{Forests}\label{sec:forests}

Trees record the ``many-to-one'' operations present in an operad, but other approaches to operads use \emph{forests}, which consist of disjoint unions of finitely many trees. The first incarnation of this idea appears in the categories of operators described by May and Thomason \cite{may_uniqueness_1978} and forms the foundation for quasi-categorical definitions to operads \cite{HA} and their comparison with the dendroidal approach \cite{HHM:16}. In the equivariant context, we are naturally led to consider forests, since the the output of an operation need no longer be single point, but rather an orbit $G/H$; if the orbit is nontrivial, then we need several root edges.

While heuristically a forest may be pictured as a collection of trees placed next to one another, more formally we can define the category of forests as a Grothendieck construction of the category $\Omega$ of trees.  The following definition describes the category of forests as the Grothendieck construction of the functor $\Fin^{\op} \rightarrow \Cat$  taking $\underline{n}$ to the product $\Omega^n$, where $\Fin$ denotes the category of finite sets of form $\underline{n} = \{1, \ldots, n\}$ for $n \geq 1$. This special case of the Grothendieck construction of a functor out of $\Fin$ is sometimes called a \emph{wreath product}; from that perspective the following category would be denoted by $\Fin \wr \Omega$.

\begin{definition}\label{defn:forests}
    Let $\Phi$ be the category with objects pairs $(\underline{n}, T)$ where $T=(T_i)_{i=1}^n$ is a disjoint union of trees, namely, objects of $\Omega$.  Its morphisms $(\underline{n}, T) \rightarrow (\underline{m}, S)$ are given by pairs $(\varphi, (f_i))$, where $\varphi \colon \underline{n} \rightarrow \underline{m}$ is a map in $\Fin$ and for each $1 \leq i \leq n$, $f_i \colon T_i \rightarrow S_{\varphi(i)}$ is a map in $\Omega$.
\end{definition}

\begin{remark}
    The category $\Phi$ above agrees with the category of forests described in \cite[\S 5.2]{Pereira:EquivariantDendroidalSets}, but is different from the one in \cite[\S 3.1]{HHM:16}. See \cite[Remark 5.28]{Pereira:EquivariantDendroidalSets} for more discussion and an explicit comparison.
\end{remark}

There is a fully faithful functor $\Omega\to \Phi$ that includes a tree $T$ as the forest $(\underline{1}, T)$ consisting of a single tree, so we can identify $\Omega$ as a subcategory of $\Phi$.  We thus often write $T$ for $(\underline{1}, T)$ when we consider $T$ as an object of $\Phi$ via this inclusion.

Similarly to trees, there are face and degeneracy maps of forests.  The degeneracy maps in $\Phi$ are analogous to the ones in $\Omega$, and many of the inner and outer face maps in $\Phi$ are also given by face maps of $\Omega$ applied to each component of the forest. However, there is an additional kind of outer face map specific to $\Phi$ called the \emph{root face inclusion}.  Such a map joins together some of the trees in the forest together into a single tree by grafting the roots onto a corolla; an example is given in \cref{fig:forest root face}.

\begin{figure}[h!]
    \centering
    \[
    \scalebox{0.85}{
    \begin{tikzpicture} 
    [level distance=10mm, 
    every node/.style={fill, circle, minimum size=.1cm, inner sep=0pt}, 
    level 1/.style={sibling distance=15mm}, 
    level 2/.style={sibling distance=15mm}, 
    level 3/.style={sibling distance=8mm},
    level 4/.style={sibling distance=7mm}]
    \node (tree) at (-6,0) [style={color=white}] {} [grow'=up] 
    child {node (level2) {} 
	child
	child
	child
    };
    \node (tree) at (-3,0) [style={color=white}] {} [grow'=up] 
    child {node (level2) {} 
	child
	child
    };
    \node (tree) at (-1,0) [style={color=white}] {} [grow'=up] 
    child {node (level2) {} 
	child
	child{ node (level3) {}
		child
		child
	}
    };
    \node[fill=white, label=$\hookrightarrow$] at (0.5,1) {};
    \node (tree) at (3,0) [style={color=white}] {} [grow'=up] 
    child [draw=myred] {node[color=myred] (level2) {} 
	child [draw=black] {node (level3) {} 
	child
	child
        child 
	}
	child  [draw=black] {node (level3) {} 
	child
        child 
	}
	child  [draw=black] {node (level3) {} 
	child 
	child  {node (level4) {} 
	child
	child
	}
	}
    };
    \tikzstyle{every node}=[]
    \end{tikzpicture}
    }
    \]
    \caption{A root face inclusion}
    \label{fig:forest root face}
\end{figure}
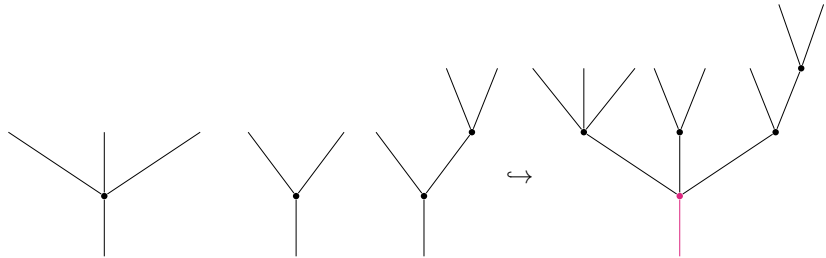
    
Finally, there is one additional kind of face map in $\Phi$ that does not change the number of vertices but rather the number of connected components of the forest, as depicted in \cref{fig:forest third face}.

\begin{figure}[h!]
    \centering
    \[
    \scalebox{0.85}{
    \begin{tikzpicture} 
    [level distance=10mm, 
    every node/.style={fill, circle, minimum size=.1cm, inner sep=0pt}, 
    level 1/.style={sibling distance=15mm}, 
    level 2/.style={sibling distance=15mm}, 
    level 3/.style={sibling distance=8mm},
    level 4/.style={sibling distance=7mm}]
    \node (tree) at (-1,0) [style={color=white}] {} [grow'=up] 
    child {node (level2) {} 
	child
	child{ node (level3) {}
		child
		child
	}
    };
    \node[fill=white, label=$\hookrightarrow$] at (0.5,1) {};
    \node (tree) at (2,0) [style={color=white}] {} [grow'=up] 
    child {node (level2) {} 
	child
	child{ node (level3) {}
		child
		child
	}
    };
    \node (tree) at (3,0) [style={color=white}] {} [grow'=up] 
    child [draw=mypurple, thick];
    \tikzstyle{every node}=[]
    \end{tikzpicture}
    }
    \]
    \caption{A third kind of face map}
    \label{fig:forest third face}
\end{figure}
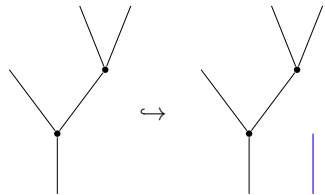

\subsection{The category $\Omega_G$}

We now give the precise definition of the category $\Omega_G$, for which we first recall some notation.  Recall from \cref{sec:Omega upper G} that, given a category $\mathcal C$ and a finite group $G$, we denote by $\mathcal C^G$ the category of $G$-objects in $\mathcal C$, defined as the category of functors $BG \rightarrow \mathcal C$ and natural transformations.  Here, we consider the examples of $\Fin^G$ and $\Phi^G$.  Let $\mathcal O_G$ be the orbit category of $G$, which we take to be the full subcategory of $\Fin^G$ on those objects for which the $G$-action is transitive.  Since we have assumed that the objects of $\Fin$ take the form $\underline{n}$, every object of $\Fin^G$ is canonically isomorphic as a $G$-set to $G/H$, where $H$ is the stabilizer of $1$. We can thus think of the objects of $\mathcal O_G$ as given by the orbits $G/H$ for subgroups $H$ of $G$, equipped with an ordering such that the coset $H$ is the least element in the ordering.

\begin{definition}\label{defn:omegasubG}
    The category $\Omega_G$ is the pullback in the diagram
    \[ \xymatrix{\Omega_G \ar[r] \ar[d] & \Phi^G \ar[d]^r \\
    \mathcal O_G \ar[r] & \Fin^G.} \]
    Here, $r \colon \Phi^G \rightarrow \Fin^G$ is the root functor taking a forest $(\underline{n}, T=(T_i))$ to the $G$-set of roots of the $T_i$. 
\end{definition}

Observe that, since we are taking $\mathcal O_G$ to be a full subcategory of $\Fin^G$, we can also regard $\Omega_G$ as a full subcategory of $\Phi^G$.  To understand $\Omega_G$, it is thus useful to describe the objects and morphisms of the category $\Phi^{G}$ more explicitly. Since an object of $\Phi^{G}$ is a functor $F\colon BG\to\Phi$, it consists of a forest $T_{1}\amalg\dots\amalg T_{n}$ in $\Phi$ together with, for each $g\in G$, an automorphism of forests 
\[ F(g) \colon T_{1}\amalg\dots\amalg T_{n}\to T_{1}\amalg\dots\amalg T_{n} \] 
describing a $G$-action on the forest. The action of $g$ gives a bijection $g \colon \underline{n} \to \underline{n}$ together with isomorphisms of trees $g\colon T_i\to T_{g\cdot i}$ in $\Omega$ for all $i\in \underline{n} = \{1,\dots, n\}$. 

A morphism in $\Phi^{G}$ is a $G$-equivariant map of forests, which is given by a pair $(\varphi, (f_i)) \colon (\underline{n}, (T_i)) \rightarrow (\underline{m}, (S_j))$ that can be thought of as map
\[ T_{1}\amalg\dots\amalg T_{n} \to S_{1}\amalg\dots\amalg S_{m} \] 
in $\Phi$ such that, for each $g\in G$, there is a commutative square
\begin{equation} \label{gactiononphimorphisms}
    \begin{tikzcd}
        T_{1}\amalg\dots\amalg T_{n}\rar["(\varphi{,}(f_i))"]\dar["g"', "\cong"] & S_{1}\amalg\dots\amalg S_{m}\dar["g", "\cong"'] \\
        T_{1}\amalg\dots\amalg T_{n}\rar["(\varphi{,}(f_i))"'] & S_{1}\amalg\dots\amalg S_{m}.
    \end{tikzcd}
\end{equation} 

Note that according to  \cref{defn:forests} the map $\varphi\colon \underline{n}\to \underline{m}$ between indexing sets is only required to be a function.  However, the additional information encoded in $\Phi^G$ imposes more structure.

\begin{lemma} \label{unpackGforests}
    For any map $(\varphi, (f_i)) \colon (\underline{n}, (T_i)) \rightarrow (\underline{m}, (S_j))$ in $\Phi^G$, the map $\varphi \colon \underline{n} \rightarrow \underline{m}$ on indexing sets is $G$-equivariant.
\end{lemma}

\begin{proof}
     If we restrict the commutative square \eqref{gactiononphimorphisms} to a single tree $T_i$ for each $i\in \underline{n}$, we obtain, on the one hand, a morphism 
    \[ T_i\xrightarrow{f_i} S_{\varphi(i)}\xrightarrow{g} S_{g\cdot \varphi(i)}, \] 
    and, on the other hand, a morphism 
    \[ T_i\xrightarrow{g} T_{g\cdot i}\xrightarrow{f_{g\cdot i}} S_{\varphi(g\cdot i)}.\] 
    By commutativity of the square, these two morphisms agree, so that $g\cdot\varphi(i)=\varphi(g\cdot i)$ for every $i\in \underline{n}$, as required.
\end{proof}

\begin{remark} 
    We can now make an important observation about the $G$-forests that occur as objects in $\Omega_G$.  An object $F$ in $\Omega_G$ is a $G$-forest such that the action of $G$ on $F$ is transitive on the roots.  In particular, every constituent tree in the forest $F$ is in the same $G$-orbit and hence all the trees in $F$ are non-equivariantly isomorphic to each other.
\end{remark} 

\subsection{Genuine $G$-trees as a Grothendieck construction}\label{sec:Omega lower G is Groth}

In this section we provide a description of $\Omega_G$ as a Grothendieck construction of a functor $\mathcal{O}_G^{\op}\to \Cat$, following a claim in \cite{Pereira:EquivariantDendroidalSets}. This description allows us to prove our last main result, \cref{theresult}, that expresses $\Omega_G$ as an iterated Grothendieck construction relating it to the categories of trees from \cref{def:T(A)}. 

We would like to define a functor $\mathcal{O}_G^{\op}\to \Cat$ that takes an object $G/H$ of $\mathcal{O}_G^{\op}$ to the category $\Omega^H$ of trees with $H$-action.  However, this approach is not well-suited for our purposes because defining functors $\Omega^K\to \Omega^H$ associated to a map of $G$-sets $G/H\to G/K$ for $H\leq K$ requires making choices of coset representatives for $K/H$.  Thus, before defining this functor we show in \cref{prop:OmegaH vs OmegaGH} that there is an alternate description of $\Omega^H$, following ideas of \cite{Pereira:EquivariantDendroidalSets}, that avoids this technical issue. 

\begin{definition}\label{DavidTODO}
    Let $\underline{G/H}$ denote the groupoid whose objects are cosets $gH\in G/H$ and whose morphisms are given by left translation
    \[ \overline{x}\colon gH\to xgH \] 
    for all $x,g\in G$. Define $\Omega^{\underline{G/H}} := \Fun(\underline{G/H}, \Omega)$. 
\end{definition}

Explicitly, an object of $\Omega^{\underline{G/H}}$ is specified by a collection of objects $T_{gH}$ of $\Omega$, for each $gH\in G/H$, together with maps of trees $\overline{x}\colon T_{gH} \to T_{xgH}$ for each $x\in G$ and $gH\in G/H$.  In particular, for all $g,g'\in G$, we have that $T_{gH}\cong T_{g'H}$ as objects in $\Omega$. Additionally, $T_{eH}$ has a well-defined $H$-action, so is an object of $\Omega^H$; more generally, $T_{gH}$ has a $gHg^{-1}$-action. The following result formalizes these observations into an equivalence of categories.

\begin{proposition}\label{prop:OmegaH vs OmegaGH}
    For $H\leq G$, there is an equivalence of categories $\Omega^{\underline{G/H}}\xrightarrow{\simeq} \Omega^H$.
\end{proposition}

\begin{proof}
    There is a functor $BH\to \underline{G/H}$ that sends the object $\ast$ to the coset $eH$ and a morphism $h\in H$ to the automorphism $\overline{h}\colon eH\to eH$. It is straightforward to check that this functor is an equivalence of categories, so the result follows by applying $\Fun(-, \Omega)$ to this equivalence. 
\end{proof}

\begin{remark}\label{remark: OmegaH vs OmegaGH is natural}
     Any inclusion $H\leq K$, induces a functor $BH\to BK$, as well as a functor $\underline{G/H}\to \underline{G/K}$, where the latter is defined on objects by sending $gH$ to $gK$.  The equivalence in the above proposition is natural, in the sense that the resulting square
    \[
        \begin{tikzcd}
            \Omega^{\underline{G/K}}\ar[d] \ar[r,"\simeq"]  &  \Omega^{K} \ar[d]\\
            \Omega^{\underline{G/H}} \ar[r,"\simeq"]  & \Omega^{H}
        \end{tikzcd}
    \]
    commutes.
\end{remark}

In fact, $\Omega^{\underline{( \ )}}\colon \mathcal O_G^{\op} \rightarrow \Cat$ defines a diagram of categories, where a map $q\colon G/H\to G/K$ in $\mathcal O_G$ induces a functor $\underline{q}\colon \underline{G/H}\to \underline{G/K}$, and precomposing with this functor yields $q^*\colon \Omega^{\underline{G/K}}\to \Omega^{\underline{G/H}}$.  The following result was stated by Pereira in \cite[Proposition 5.47]{Pereira:EquivariantDendroidalSets}, together with a sketch of the argument.  

\begin{proposition}\label{prop:OmegaGH Grothendieck}
    There is an equivalence of categories $\int^{\mathcal O_G^{\op}} \Omega^{\underline{( \ )}} \simeq \Omega_G$.
\end{proposition}

\begin{proof} 
    To define a functor $F\colon \int^{\mathcal O_G^{\op}} \Omega^{\underline{( \ )}} \to\Omega_G$, we need, for each $H\leq G$, a specified ordering on $G/H=\{g_1H,\dots,g_nH\}$ with $g_1=e$, so that there is an isomorphism between $G/H$ and a specified object of $\Fin^G$.

    We now unpack the data of the objects of the domain. Applying \cref{defn:other Groth construction} to the functor $\Omega^{\underline{( \ )}}$, we see that the objects of the Grothendieck construction are pairs $(G/H, T\colon \underline{G/H}\to \Omega)$, where $H\leq G$ and $T$ is a functor. The functor $F$ sends $(G/H,T)$ to the $G$-forest  $T_{g_1H}\amalg \dots \amalg T_{g_n H}$, which is an object of $\Omega_G$ as we now explain. Given $T$, for each coset $gH\in G/H$ there is an associated tree $T_{gH}$. By construction, the collection of all these trees assembles to a forest $T_{g_1H}\amalg \dots \amalg T_{g_n H}$ whose trees are indexed by the cosets. Moreover, for each morphism $\overline{x}\colon gH\to xgH$ in $\underline{G/H}$, namely, for each $x\in G$, there is a morphism between the corresponding trees $T(\overline{x})\colon T_{gH}\to T_{xgH}$. These data assemble into an automorphism of the forest $T_{g_1H}\amalg \dots \amalg T_{g_n H}$ that corresponds to a $G$-action, and via this action the collection of root edges is isomorphic to the transitive $G$-set $G/H$. 

    To define the functor $F$ at the level of morphisms, recall that the morphisms of $\int^{\mathcal O_G^{\op}}\Omega^{\underline{(\ )}}$ are of the form 
    \[ (q, \alpha)\colon (G/H, T\colon\underline{G/H}\to\Omega)\to (G/K, S \colon \underline{G/K} \to \Omega), \] 
    where $q\colon G/H\to G/K$ is a $G$-equivariant map and $\alpha\colon T\Rightarrow \Omega^{\underline{q}}(S)$ is a natural transformation. Note that $\Omega^{\underline{q}}(S)\colon \underline{G/H}\to \Omega$ is the functor that takes an object $gH$ to the tree $S_{q(gH)}$, and a morphism $\overline{x}\colon gH\to xgH$ to the morphism of trees 
    \[ S(\overline{x})\colon S_{q(gH)}\to S_{xq(gH)}=S_{q(xgH)}. \] 
    Thus, the natural transformation $\alpha$ has components given by maps of trees $\alpha_{gH}\colon T_{gH}\to S_{q(gH)}$ for each $gH\in G/H$. The naturality of $\alpha$ encodes the fact that, for every $x\in G$, there is a commutative square
    \[\begin{tikzcd}
        T_{gH}\rar["\alpha_{gH}"]\dar["\cong","x"'] & S_{q(gH)}\dar["x", "\cong"']\\
        T_{xgH}\rar["\alpha_{xgH}"'] & S_{q(xgH).}
    \end{tikzcd}\] 
    Following the description of morphisms in $\Omega_G$ just above \cref{unpackGforests}, we can define $F(q,\alpha)$ as the map of forests $(\varphi,(\alpha_{g_iH}))$ where $\varphi\colon \underline{n} \to \underline{m}$ is the map induced by $q\colon G/H \to G/K$ and the chosen ordering of these orbits. Note that functoriality is immediate from this construction.

     Using the explicit description of $\Phi^G$ in \cref{unpackGforests}, each object in the category $\Omega_G$ of \cref{defn:omegasubG} is, by definition, isomorphic to $F(G/H,T)$ for some $(G/H,T)$; hence $F$ is essentially surjective. To prove $F$ is faithful, note that $\varphi$ is uniquely determined by $q$ and the maps of trees are then determined by $\alpha$. 
     Given $(\varphi,(f_i))\colon F(G/H,T)\to F(G/K,S)$, we have that $\varphi\colon \underline{n} \to \underline{m}$ and the orderings determine a unique $q\colon G/H\to G/H$ that is equivariant by \cref{unpackGforests}. Moreover, the maps $f_i\colon T_{g_iH} \to S_{g_{\varphi(i)}K}$ assemble into a natural transformation $\alpha \colon T \Rightarrow \Omega^{\underline{q}}(S)$ with $\alpha_{g_iH}=f_i$, where naturality is given by the commutativity of diagram (\ref{gactiononphimorphisms}).
\end{proof}

Recall from \cref{GOmega is GTree} that there is an equivalence of categories between $\Omega^H$ and a Grothendieck construction on categories of labeled $H$-trees; we now establish an analogous result for $\Omega^{\underline{G/H}}$, using categories of trees labeled by $G$-sets over $G/H$. We first introduce the relevant definitions.

\begin{definition}
    For $H\leq G$, let $\mathscr{F}_{G/H,*}:= ~_{G/H}\!\setminus (\Fin_G) /_{G/H}$ denote the category in which an object is a $G$-retractive finite $G$-set over $G/H$, \emph{i.e.}, a diagram of finite $G$-sets $G/H\to A\to G/H$ so that the composite is the identity.  Morphisms in this category are maps of finite $G$-sets relative to $G/H$.  We henceforth write $\mathscr{F}^{\op}_{G/H,*}$ for $\left(\mathscr{F}_{G/H,*}\right)^{\op}$.
\end{definition}

An object of $\mathscr{F}_{G/H,*}$ is a $G$-retractive finite $G$-set $A'\leftrightarrows G/H$ for some $H\leq G$. Note that a retractive $G$-set $A'$ is canonically isomorphic, as a $G$-set over $G/H$, to $A\amalg G/H$, where $A$ is the complement of the image of the section $G/H\to A'$. We sometimes abuse notation by writing $A'=A_+$ or by only making reference to the retraction $A_+\to G/H$. We also sometimes write $\alpha_+\colon A_+\to G/H$ for the retraction and $\alpha\colon A\to G/H$ for its restriction to $A$.

\begin{definition}
    Let $\alpha_+\colon A_+\to G/H$ be an object of $\mathscr{F}_{G/H, *}$. Define $\T^{\underline{G/H}}(A, \alpha)$  to be the category whose objects are pairs $(T, \chi)$ with $T$ an object of $\Omega^{\underline{G/H}}$ and $\chi\colon A\to \ell(T)$ a $G$-equivariant bijection. The equivariant bijection $\chi$ must be compatible with the root labels on $T$ in the following sense: for all $a\in A$, the leaf labeled by $a$ is on the tree $T_{\alpha(a)}$. A morphism $f\colon (T, \chi_T) \to (S, \chi_S)$ is a map in $\Omega^{\underline{G/H}}$ that preserves the labels of the leaves and the roots, thus in particular it consists of a family of maps of trees $f_{gH} \colon T_{gH} \to S_{gH}$ that preserve the labels of the leaves and are compatible with the $G$-action. 
\end{definition}

There is a forgetful functor $\T^{\underline{G/H}}(A, \alpha)\to \Omega^{\underline{G/H}}$ which forgets the labeling.  In analogy with the equivalence $\Omega^{\underline{G/H}}\simeq \Omega^H$ from \cref{prop:OmegaH vs OmegaGH} we have the following result.

\begin{proposition}\label{trees on fibers}
    For each $\alpha_+\colon A_+\to G/H$,  there is an equivalence of categories $\T^{\underline{G/H}}(A, \alpha) \simeq \T^H(\alpha^{-1}(eH))$.
\end{proposition}

\begin{proof}
    We define a functor $F\colon \T^{\underline{G/H}}(A, \alpha)\to \T^H(\alpha^{-1}(eH))$ as follows. Given an object $(T, \chi_T)\in \T^{\underline{G/H}}(A, \alpha)$, we let $F(T,\chi_T)=(T_{eH}, \chi_{T_{eH}})\in \T^H(\alpha^{-1}(eH))$, where $\chi_{T_{eH}}\colon \alpha^{-1}(eH)\to \ell(T_{eH})$ is the restriction of $\chi_T\colon A\to \ell(T)$ to $\alpha^{-1}(eH)\subseteq A$. By definition, a morphism $f\colon (T,\chi_T)\to (S,\chi_S)$ in $\T^{\underline{G/H}}(A, \alpha)$ restricts to a map of trees $f_{eH}\colon T_{eH} \to S_{eH}$, which is moreover compatible with the labelings. Hence we may define $F(f) = F(f_{eH})$, and functoriality is straightforward to check. To prove $F$ is an equivalence of categories we check that it is full, faithful, and essentially surjective. For the remainder of this proof we fix objects $(T,\chi_T)$ and $(S,\chi_S)$ in $\T^{\underline{G/H}}(A, \alpha)$.

    For fullness, suppose that $f\colon T_{eH}\to S_{eH}$ is a morphism between $F(T,\chi_T)$ and $F(S,\chi_S)$.  Define a map $\widehat{f}\colon (T,\chi_T)\to (S,\chi_S)$ where $\widehat{f}_{gH}\colon T_{gH}\to S_{gH}$ is the composite
    \[
        T_{gH}\xrightarrow{g^{-1}} T_{eH}\xrightarrow{f} S_{eH}\xrightarrow{g} S_{gH}
    \]
    where the first and last map are the action of $G$ on the forest.  Compatibility of this map with the $G$-actions is straightforward to check, using the fact that $f$ is $H$-equivariant. For compatibility with the labeling bijections, let $a\in \alpha^{-1}(gH)\subset A$, and observe that we have
    \[
        \widehat{f}(\chi_{T_{gH}}(a)) =gf(g^{-1}\chi_{T_{gH}}(a)) = gf(\chi_{T_{eH}}(g^{-1}a)) = g(\chi_{S_{eH}}(g^{-1}a)) = \chi_{S_{gH}}(a)
    \]
    where the first equality is the definition of $\widehat{f}$, the second and fourth are the equivariance of $\chi_T$ and $\chi_S$, and the third is the fact that $f$ is a map in $\T^G(\alpha^{-1}(eH))$. It follows that $\widehat{f}$ is indeed a map in $\T^{\underline{G/H}}(A,\alpha)$ and, by definition, we have $F(\widehat{f})=\widehat{f}_{eH} = f$ so $F$ is a full functor. 
    
    To see that $F$ is faithful, note that we have a commutative square
    \[\begin{tikzcd}[ampersand replacement=\&]
	{\T^{\underline{G/H}}(A, \alpha)} \& {\T^H(\alpha_0^{-1}(eH))} \\
	{\Omega^{\underline{G/H}}} \& {\Omega^H}
	\arrow["F", from=1-1, to=1-2]
	\arrow[from=1-1, to=2-1]
	\arrow[from=1-2, to=2-2]
	\arrow["\simeq", from=2-1, to=2-2]
    \end{tikzcd}\]
    where the vertical maps are the forgetful functors and the the equivalence is the map from \cref{prop:OmegaH vs OmegaGH}.  Since the forgetful functors and the equivalence are faithful, we see that $F$ must be faithful.

    Finally, we must show that $F$ is essentially surjective. Given $(R,\chi_R)\in \T^H(\alpha^{-1}(eH))$, we build an object $(\widehat{R},\chi_{\widehat{R}})\in \T^{\underline{G/H}}(A,\alpha)$, where $\widehat{R}$ is the forest obtained by ``induction'' on the $H$-tree $R$.  Explicitly, we pick a collection of coset representatives $1=\gamma_1, \gamma_2,\dots , \gamma_n$ for $G/H$ and let
    \[
        \widehat{R} = \coprod_{i=1}^n R_{\gamma_i}
    \]
    where $R_{\gamma_i}=R$ for all $i$.  For any $g\in G$, and any $i$ there exist unique $j_g\in 1,\dots, n$ and $h_{i}(g)\in H$ such that $g\gamma_i = \gamma_{j_g}h_{i}(g)$, by the definition of coset representatives. The edges of $\widehat{R}$ can each be uniquely written as pairs $(e,\gamma_i)$ where $e\in E(R)$ and the action of $G$ on $R$ is defined on edges by $g\cdot (e,\gamma_i) = (h_i(g)e,\gamma_{j_g})$; the definition on vertices is essentially the same. To define the labeling bijection $\chi_{\widehat{R}}$, we use the fact that induction gives an equivalence of categories between finite $H$-sets and finite $G$-sets over $G/H$.  In particular, through this bijection, $\alpha\colon A\to G/H$ is in canonical equivariant bijection (over $G/H$) with the induction of $\alpha^{-1}(eH)$.  In particular, we have
    \[
        A\cong G\times_H \alpha^{-1}(eH)\xrightarrow{G\times_H \chi_R} G\times_H \ell(R)\cong \coprod_{i=1}^n\ell(R)=\ell(\widehat{R})
    \]
    which gives the equivariant labeling bijection $\chi_{\widehat{R}}$.  By construction $F(\widehat{R},\chi_{\widehat{R}}) = (R,\chi_R)$.
\end{proof}

In fact, this equivalence is suitably natural to induce an equivalence on Grothendieck constructions, as we now describe.

\begin{construction}\label{constr:oplax TGH}
    Just as in \cref{constr:oplaxTG}, the assignment 
    \[ (A_+\xrightarrow{\alpha_+} G/H)\mapsto \T^{\underline{G/H}}(A, \alpha) \] 
    extends to an oplax functor 
    \[ \T^{\underline{G/H}}\colon \mathscr{F}_{G/H,  *}^{\op}\to \Cat. \] 
    To describe $\T^{\underline{G/H}}$ on morphisms, let $\varphi\colon (A'_+\xrightarrow{\alpha'_+} G/H) \to (A_+\xrightarrow{\alpha_+} G/H)$ be a map in $\mathscr{F}_{G/H, *}$ and $(\amalg_{G/H} T_{gH},\chi)$ be an object of $\T^{\underline{G/H}}(A, \alpha)$. Recall that in \cref{constr:oplaxTG} we described a process that attaches corollas according to $\varphi$ to produce an $A'$-labeled tree from an $A$-labeled tree. There is an analogous operation on forests, which we also denote by $\varphi^*$; explicitly, $\varphi^*(\amalg_{G/H} T_{gH})$ is the forest given by $\amalg_{G/H} (\varphi|_{\alpha^{-1}(gH)})^*(T_{gH})$, where $\varphi|_{\alpha^{-1}(gH)}\colon \alpha^{-1}(gH) \to \varphi(\alpha^{-1}(gH))$ is the non-equivariant map that witnesses $\varphi$ over the fiber of $gH\in G/H$. The labeling of this new forest is given by the $G$-equivariant bijection $\chi\varphi$.
\end{construction}

We can then apply the same proof as for \cref{TG is oplax} to obtain the following result.

\begin{proposition}\label{prop:TGH oplax}
    \Cref{constr:oplax TGH} defines an oplax functor.
\end{proposition}

Thus, we may consider its Grothendieck construction, for which we first show the following result.

\begin{lemma}\label{lem:TH vs TGH}
   For each $H\leq G$, there is an equivalence of Grothendieck constructions
   \[ \int_{\mathscr{F}_{G/H, *}^{\op}} \T^{\underline{G/H}} \simeq \int_{\Fin_{H,*}^{\op}} \T^H. \] 
\end{lemma}

\begin{proof}
    For each $H\leq G$, the functor ${\rm fib}_{eH}\colon \mathscr{F}_{G/H, *}^{\op}\to \Fin_{H,*}^{\op}$, which takes the fiber over $eH$, is an equivalence of categories. By \cref{Grothendieck indexing cat change}, there is an equivalence of categories
    \[ \int_{\mathscr{F}_{G/H, *}^{\op}} \T^H\circ {\rm fib}_{eH} \xrightarrow{\simeq} \int_{\Fin_{H,*}^{\op}} \T^H. \]
    Additionally, by \cref{remark: OmegaH vs OmegaGH is natural} and \cref{trees on fibers}, there is a natural transformation of functors $\T^{\underline{G/H}}\Rightarrow \T^H \circ {\rm fib}_{eH}$ that is an equivalence of categories for each labeling $G$-set $A_+\to G/H$. Then
    \[ \int_{\mathscr{F}_{G/H, *}^{\op}} \T^{\underline{G/H}} \xrightarrow{\simeq} \int_{\mathscr{F}_{G/H, *}^{\op}} \T^H\circ {\rm fib}_{eH} \] 
    is an equivalence of categories by \cref{equivalence on Grothendiecks}.
\end{proof}

\begin{remark}
    Let $H$ be a subgroup of $G$. An object in the Grothendieck construction $\int_{\mathscr{F}_{G/H, *}^{\op}} \T^{\underline{G/H}}$ can be described explicitly as a pair $(A_+\xrightarrow{\alpha_+} G/H, (T,\chi))$, where $$(T = \amalg_{G/H} T_{gH}, \chi\colon A\to\ell(T))$$ is an object of $\T^{\underline{G/H}}(A,\alpha)$. A morphism 
    \[
        (\varphi,f)\colon(A_+\xrightarrow{\alpha_+} G/H,(\coprod_{G/H} T_{gH},\chi_T))\to (B_+\xrightarrow{\beta_+} G/H,(\coprod_{G/H} S_{gH},\chi_S))
    \]
    consists of an equivariant function $\varphi\colon B_+\to A_+$ over and under $G/H$, together with a map $f\colon \varphi^*(T)\to S$ in $\T^{\underline{G/H}}(B,\beta)$. Since $\varphi^*(T)=\amalg_{G/H}(\varphi|_{\alpha^{-1}(gH)})^*T_{gH}$, the map $f$ is given by a collection of maps $f_{gH}\colon (\varphi|_{\alpha^{-1}(gH)})^*T_{gH} \to S_{gH}$ compatible with the $G$-action and the labeling. 
\end{remark}

\begin{example}
    Let $G=C_4 = \{1, i, -1, -i\}$, $H=C_2 = \{\pm 1\}$, and $q\colon C_4/e\to C_4/C_2$ be the canonical quotient map given by $q(gC_4) = gC_2$. There is a morphism $(q,(T, \chi_T)) \rightarrow (\id_{C_4/C_2},(S, \chi_S))$ in $\int_{\mathscr{F}_{G/H, *}^{\op}} \T^{\underline{G/H}}$
    \[
    \scalebox{0.85}{
    \begin{tikzpicture} 
    [level distance=10mm, 
    every node/.style={fill, circle, minimum size=.1cm, inner sep=0pt}, 
    level 1/.style={sibling distance=15mm}, 
    level 2/.style={sibling distance=15mm}, 
    level 3/.style={sibling distance=8mm},
    level 4/.style={sibling distance=7mm}]
    \node (tree) at (-2,0) [style={color=white}] {} [grow'=up] 
    child[draw=myred, thick] {node[color=myred] (y) (level2) {} 
	child
    };
    \node (tree) at (-1,0) [style={color=white}] {} [grow'=up] 
    child[draw=myred, thick] {node[color=myred] (level2) {} 
	child
    };
    \node[fill=none] (o) at (0,0) {};
    \node[fill=none] at ($(o) + (-2, -0.2)$){\textcolor{myred}{$C_2$}};
    \node[fill=none] at ($(o) + (-1, -0.2)$){\textcolor{myred}{$iC_2$}};
    \node[fill=none] at ($(o) + (-2, 2.2)$){\textcolor{myred}{$C_2$}};
    \node[fill=none] at ($(o) + (-1, 2.2)$){\textcolor{myred}{$iC_2$}};
    \node[fill=none] at ($(o) + (40mm, 10mm)$) {$(S, \chi_S)$};
    
    \node[fill=none, label=$\longrightarrow$] at (0,0.5) {};
    \node (tree) at (3,0) [style={color=white}] {} [grow'=up] 
    child[draw=myred, thick] {node[color=myred] (level2) {} 
	child[draw=myblue, thick] 
	child[draw=myblue, thick] 
    };
    \node (tree) at (1,0) [style={color=white}] {} [grow'=up] 
    child[draw=myred, thick] {node[color=myred] (level2) {} 
	child[draw=myblue, thick] 
	child[draw=myblue, thick] 
    };
    \node[fill=none] at ($(o) + (1, -0.2)$){\textcolor{myred}{$C_2$}};
    \node[fill=none] at ($(o) + (3, -0.2)$){\textcolor{myred}{$iC_2$}};
    \node[fill=none] at ($(o) + (1.6, 2.2)$){\textcolor{myblue}{$-1$}};
    \node[fill=none] at ($(o) + (0.25, 2.2)$){\textcolor{myblue}{$1$}};
    \node[fill=none] at ($(o) + (3.6, 2.2)$){\textcolor{myblue}{$-i$}};
    \node[fill=none] at ($(o) + (2.25, 2.2)$){\textcolor{myblue}{$i$}};
    \node[fill=none] at ($(o) + (-30mm, 12mm)$){$(T, \chi_T)$};
    \end{tikzpicture}
    }
    \]
    with the source object in $\T^{\underline{C_4/C_2}}(C_4/e, q)$ and the target object in $\T^{\underline{C_4/C_2}}(C_4/C_2, \id)$,
    given by the pair $(q_+, f)$ with $q_+\colon (C_4/e)_+ \to (C_4/C_2)_+$ and $f$ the map of $C_4/e$-labeled $C_4$-trees
        \[
    \scalebox{0.85}{
    \begin{tikzpicture} 
    [level distance=10mm, 
    every node/.style={fill, circle, minimum size=.1cm, inner sep=0pt}, 
    level 1/.style={sibling distance=15mm}, 
    level 2/.style={sibling distance=15mm}, 
    level 3/.style={sibling distance=8mm},
    level 4/.style={sibling distance=7mm}]
    \node (tree) at (-3,0) [style={color=white}] {} [grow'=up] 
    child[draw=myred, thick] {node[color=myred] {}
    child[draw=myred, thick] {node[color=myred] (x) (level2) {} 
	child[draw=myblue, thick] 
	child[draw=myblue, thick]
    }
    };
    \node (tree) at (-1,0) [style={color=white}] {} [grow'=up] 
    child[draw=myred, thick] {node[color=myred] {}
    child[draw=myred, thick] {node[color=myred] (level2) {} 
	child[draw=myblue, thick] 
	child[draw=myblue, thick] 
    }
    };
    \node[fill=none] at ($(x) + (0mm, -12mm)$){\textcolor{myred}{$C_2$}};
    \node[fill=none] at ($(x) + (20mm, -12mm)$){\textcolor{myred}{$iC_2$}};
    \node[fill=none] at ($(x) + (0.25, 2.2)$){\textcolor{myblue}{$-1$}};
    \node[fill=none] at ($(x) + (-0.35, 2.2)$){\textcolor{myblue}{$1$}};
    \node[fill=none] at ($(x) + (2.25, 2.2)$){\textcolor{myblue}{$-i$}};
    \node[fill=none] at ($(x) + (1.65, 2.2)$){\textcolor{myblue}{$i$}};
    \node[fill=none] at (-45mm, 10mm){$(q_+^*T, q^*\chi_T)$};
    \node[fill=none, label=$\longrightarrow$] at (0,0.5) {};
    \node (tree) at (1,0) [style={color=white}] {} [grow'=up] 
    child[draw=myred, thick] {node[color=myred] (y) (level2) {} 
	child[draw=myblue, thick] 
	child[draw=myblue, thick] 
    };
    \node (tree) at (3,0) [style={color=white}] {} [grow'=up] 
    child[draw=myred, thick] {node[color=myred] (level2) {} 
	child[draw=myblue, thick] 
	child[draw=myblue, thick] 
    };
    \node[fill=none] at ($(x) + (40mm, -12mm)$){\textcolor{myred}{$C_2$}};
    \node[fill=none] at ($(x) + (60mm, -12mm)$){\textcolor{myred}{$iC_2$}};
    \node[fill=none] at ($(1,1) + (6mm, 12mm)$){\textcolor{myblue}{$-1$}};
    \node[fill=none] at ($(1,1) + (-8mm, 12mm)$){\textcolor{myblue}{$1$}};
    \node[fill=none] at ($(1,1) + (26mm, 12mm)$){\textcolor{myblue}{$-i$}};
    \node[fill=none] at ($(1,1) + (12mm, 12mm)$){\textcolor{myblue}{$i$}};
    \node[fill=none] at (45mm,10mm) {$(S, \chi_S)$};
    \end{tikzpicture}
    }
    \] which removes the two $C_4/C_2$-orbits of unary vertices.
    \end{example}

\begin{proposition} 
    The assignment $G/H\mapsto \int_{\mathscr{F}_{G/H, *}^{\op}} \T^{\underline{G/H}}$ assembles into a pseudofunctor 
    \[
        \int_{\mathscr{F}_{(\ ), *}^{\op}} \T^{\underline{(\ )}} \colon \mathcal{O}_G^{\op}\to \Cat.
    \]
\end{proposition}

\begin{proof} 
    We first describe the action on morphisms of $\mathcal{O}_G^{\op}$. Let $q\colon G/H\to G/K$ be an equivariant function.  There is an induced functor
    \[
        q^*\colon \int_{\mathscr{F}_{G/K, *}^{\op}} \T^{\underline{G/K}}\to \int_{\mathscr{F}_{G/H, *}^{\op}} \T^{\underline{G/H}}
    \]
    that sends an object $(A_+\xrightarrow{\alpha_+} G/K,T)$ to $(q^*A_+ \xrightarrow{q^*\alpha_+} G/H,q^*T)$, where $q^*\alpha_+$ fits into a pullback square
    \[
        \begin{tikzcd}
            q^*A_+\ar[r] \ar[d,"q^*\alpha_+"'] & A_+ \ar[d,"\alpha_+"]\\
            G/H\ar[r,"q"'] & G/K
        \end{tikzcd}
    \]
    and $q^*T$ is given by the composite functor $\underline{G/H}\xrightarrow{q}\underline{G/K} \xrightarrow{T} \Omega$. For the sake of the definition, we specify a choice of pullback as the usual $G$-subset of the cartesian product with diagonal action, but we note that the choice ultimately does not matter, as we are constructing a pseudofunctor.
    
    Thus, explicitly $(q^*T)_{gH} = T_{q(gH)}$, and we view $q^*T$ as an object in $\T^{\underline{G/H}}(q^*A_+,q^*\alpha)$ via the labeling bijection between $q^*A_+$ and the leaves of $q^*T$ that is uniquely determined by the bijection between $A$ and the leaves of $T$. The action of $q^*$ on morphisms is also defined via pullback.

    Although taking pullback along $\id$ is not the identity on the nose, it is naturally isomorphic to it. If $q\colon G/H\to G/K$ and $p\colon G/K\to G/L$ are composable morphisms, then  $(p\circ q)^*$ and $q^*\circ p^*$ need not be equal.  Rather, they are canonically naturally isomorphic, via the universal property of the pullback. One can show that this data assembles into a pseudofunctor.
\end{proof}

We now have all the necessary ingredients to express $\Omega_G$ as an iterated Grothendieck construction,  
via the proof of the following result.   

\begin{theorem}\label{theresult}
    There is an equivalence of categories
    \[
    \Omega_G \simeq \int^{\mathcal O_G^{\op}} \left(\int_{\mathscr{F}_{(\ ), *}^{\op}} \T^{\underline{(\ )}} \right).
    \]
\end{theorem}

We prove this result by comparing the Grothendieck construction of $\T^{\underline{(\ )}}$ on the right-hand side to the one on $\Omega^{\underline{(\ )}}$ from \cref{prop:OmegaGH Grothendieck}. To do so, we first establish a pointwise comparison between these constructions, defining a natural transformation $\eta$ between $G/H\mapsto \int_{\mathscr{F}_{G/H, *}^{\op}}\T^{\underline{G/H}}$ and $G/H\mapsto \Omega^{\underline{G/H}}$ that assembles into the desired equivalence of categories.

\begin{definition}
    For each $H\leq G$, define 
    \[ \eta_{G/H}\colon \int_{\mathscr{F}_{G/H, *}^{\op}}\T^{\underline{G/H}} \to \Omega^{\underline{G/H}} \] 
    similarly as in \cref{GOmega is GTree}. That is, an object $(A_+\xrightarrow{\alpha_+} G/H, \amalg_{G/H} T_{gH})$ is sent to the underlying forest $\amalg_{G/H} T_{gH}$, and a morphism 
    \[(\varphi,f)\colon (A_+\xrightarrow{\alpha_+} G/H, \coprod_{G/H} T_{gH}) \to (B_+\xrightarrow{\beta_+} G/H, \coprod_{G/H} S_{gH})\]
    is sent to the composite
    \[ \coprod_{G/H} T_{gH} \to \varphi^*(\coprod_{G/H} T_{gH}) \xrightarrow{f} \coprod_{G/H} S_{gH} \] 
    in $\Omega^{\underline{G/H}}$, where $\amalg_{G/H} T_{gH} \to \varphi^*(\amalg_{G/H} T_{gH})$ is the unique composite of outer face maps that attaches corollas.
\end{definition}

Note that since 
\[ \varphi^*(\coprod_{G/H} T_{gH})=\coprod_{G/H} (\varphi |_{\alpha^{-1}(gH)})^*(T_{gH}), \] 
the composite above is the coproduct over $G/H$ of the composites
\[ T_{gH} \xrightarrow{\iota_{T_{gH}}}  (\varphi |_{\alpha^{-1}(gH)})^*(T_{gH}) \xrightarrow{f_{gH}} S_{gH}. \] 

\begin{lemma}\label{etaisequiv}
    For each $H\leq G$, the functor $\eta_{G/H}$ is an equivalence of categories.
\end{lemma}

\begin{proof}
   The natural transformation $\eta$ is defined so that the following diagram commutes: 
   \[ \begin{tikzcd}
    \Omega^{\underline{G/H}} \ar[d, swap, "\simeq"] & \int_{\Fin_{G/H,*}^{\op}} \T^{\underline{G/H}}\ar[l, swap, "\eta_{G/H}"] \ar[d, "\simeq"] \\
    \Omega^H  & \int_{\Fin_{H,*}^{\op}} \T^{H} \ar[l,"\simeq"]
    \end{tikzcd} \] 
    where the maps labeled $\simeq$ are the equivalences from \cref{prop:OmegaH vs OmegaGH}, \cref{GOmega is GTree}, and \cref{lem:TH vs TGH}. The result then follows by the 2-out-of-3 property.
\end{proof}

\begin{proposition}\label{etaisnatural}
    The functors $\eta_{G/H}$ assemble into a pseudonatural transformation 
    \[\eta\colon \int_{\mathscr{F}_{(-), *}^{\op}}\T^{\underline{(-)}} \Rightarrow  \Omega^{\underline{(-)}}\] 
    of functors $\mathcal{O}_G^{\op} \to \Cat$.
\end{proposition}

\begin{proof} 
    Let $q\colon G/H\to G/K$ be a morphism in $\mathcal O_G$. We will prove that the diagram
    \[ \begin{tikzcd}
        \int_{\mathscr{F}_{G/K, *}^{\op}}\T^{\underline{G/K}}\ar[r, "\eta_{G/K}"] \ar[d] & \Omega^{\underline{G/K}} \ar[d] \\
        \int_{\mathscr{F}_{G/H, *}^{\op}}\T^{\underline{G/H}} \ar[r, swap, "\eta_{G/H}"] & \Omega^{\underline{G/H}}
    \end{tikzcd} \] 
    commutes. Given an object $(A_+\to G/K, \amalg_{G/K} T_{gK})$ in $\int_{\mathscr{F}_{G/K, *}^{\op}}\T^{\underline{G/K}}$, taking the composite along the top and right yields
    \[(A_+\to G/K, \coprod_{G/K} T_{gK})\mapsto \coprod_{G/K} T_{gK}\mapsto \coprod_{G/H} T_{q(gH)},\] 
    while taking the composite along the left and bottom yields
    \[(A_+ \to G/K, \coprod_{G/K} T_{gK})\mapsto (q^*(A)_+\to G/H, \coprod_{G/H} T_{q(gH)})\mapsto \coprod_{G/H} T_{q(gH)}.\] 
    Hence the diagram commutes at the level of objects.

    Next, suppose we are given a morphism 
    \[ (\varphi,f)\colon (A_+ \xrightarrow{\alpha_+} G/K, \coprod_{G/K} T_{gK}) \to (B_+ \xrightarrow{\beta_+} G/K, \coprod_{G/K} S_{gK}), \] 
    consisting of an equivariant map $\varphi\colon B_+\to A_+$ over and under $G/K$ and a map of forests 
    \[
      f =  \coprod_{G/K} f_{gK}\colon  \coprod_{G/K} (\varphi|_{\alpha^{-1}(gK)})^*T_{gK} \to \coprod_{G/K} S_{gK}.
    \]
    Tracing the action on this morphism along the top-right corner of the diagram we obtain
    \begin{align*}
        (\varphi, f) & \mapsto\left( \coprod_{G/K} T_{gK} \xrightarrow{\coprod_{G/K} \iota_{T_{gK}}} \coprod_{G/K} (\varphi|_{\alpha^{-1}(gK)})^*T_{gK} \xrightarrow{\coprod_{G/K} f_{gK}} \coprod_{G/K} S_{gK}\right)\\
        & \mapsto \left( \coprod_{G/H} T_{q(gH)} \xrightarrow{\coprod_{G/H} \iota_{T_{gK}}} \coprod_{G/H} (\varphi|_{\alpha^{-1}(q(gH))})^*T_{q(gH)} \xrightarrow{\coprod_{G/H} f_{q(gH)}} \coprod_{G/H} S_{q(gH)}\right).
    \end{align*}
    On the other hand, if we trace the action along the left-bottom corner we obtain
    \begin{align*}
        (\varphi, f) &\mapsto (q^*(\varphi) \colon q^*(B)_+ \to q^*(A)_+, \coprod_{G/H} f_{q(gH)})\\
        &\mapsto \left( \coprod_{G/H} T_{q(gH)} \xrightarrow{\coprod_{G/H} \iota_{T_{gK}}} \coprod_{G/H} (\varphi|_{\alpha^{-1}(q(gH))})^*T_{q(gH)} \xrightarrow{\coprod_{G/H} f_{q(gH)}} \coprod_{G/H} S_{q(gH)}\right).
    \end{align*}
    and hence the diagram commutes at the level of morphisms as well.

    Finally, we note that we can construct the natural transformation $\eta$ with structure $2$-cells equal to the identity. To check this one needs to check the compatibility conditions with respect to the identity morphisms and composition. These follow from the observation that the associators and unitors of the pseudofunctor $\int_{\mathscr{F}_{(-), *}^{\op}}\T^{\underline{(-)}}$ only affect the choice of labelings of leaves on trees, and thus become identities after whiskering with the functors $\eta_{G/H}$, as they forget the labelings.
\end{proof}

\begin{proof}[Proof of \cref{theresult}]
    The theorem is a direct consequence of \cref{prop:OmegaGH Grothendieck}, \cref{etaisequiv}, and \cref{etaisnatural}, since it follows from \cref{equivalence on Grothendiecks} that there is an equivalence of Grothendieck constructions
    \[ \int^{\mathcal O_G^{\op}} \Omega^{\underline{(-)}} \simeq \int^{\mathcal O_G^{\op}} \left(\int_{\mathscr{F}_{(-), *}^{\op}}\T^{\underline{(-)}}\right), \] 
    as desired.
\end{proof}

\begin{remark}\label{rmk:unpack OmegaG}
    Using the definitions of the Grothendieck constructions from \cref{defn:grothconstruction} and \cref{defn:other Groth construction}, we can explicitly describe the objects and morphisms of the iterated Grothendieck construction\[
    \int^{\mathcal O_G^{\op}} \left(\int_{\mathscr{F}_{(-), *}^{\op}}\T^{\underline{(-)}}\right)
    \] as follows. An object is a triple $(G/H, A_+\xrightarrow{\alpha_+} G/H, \amalg_{G/H} T_{gH})$ for some $H\leq G$, a labeling set $A_+\xrightarrow{\alpha_+} G/H$ from  
    $\mathscr{F}_{G/H,*}$, and an object $\amalg_{G/H} T_{gH}$ of $\T^{\underline{G/H}}(A, \alpha)$. A morphism 
    \[(q, \varphi, f)\colon (G/H, A_+\xrightarrow{\alpha_+} G/H, \coprod_{G/H} T_{gH}) \to (G/K, B_+\xrightarrow{\beta_+} G/K, \coprod_{G/K} S_{gK})\] 
    consists of the following data:
    \begin{itemize}
        \item a map $q\colon G/H\to G/K$ in $\mathcal O_G$,
        
        \item a map of $G$-sets $\varphi \colon (q^*B_+\xrightarrow{q^*\beta} G/H) \to (A\xrightarrow{\alpha} G/H)$ in $\mathscr{F}_{G/H, *}$, and
        
        \item a map of forests $f\colon \varphi^*(\amalg_{G/H} T_{gH}) \to \amalg_{G/H} S_{q(gH)})$ in $\T^{\underline{G/H}}(q^*B,{q^* \beta})$.
    \end{itemize}
    Note that both the forests $\varphi^*(\amalg_{G/H} T_{gH})$ and $\amalg_{G/H} S_{q(gH)}$ are indeed objects of $\T^{\underline{G/H}}(q^*B, q^*\beta)$.
\end{remark}

\begin{remark}
    One might have hoped, in analogy with \cref{GOmega is GTree}, to express $\Omega_G$ as the Grothendieck construction of a single oplax functor $\overline{\T}\colon \mathcal{C}\to \Cat$, where $\mathcal{C}$ is a category of ``genuine finite pointed $G$-sets.'' Such a category is considered, in instance, in work of Bonventre and Bonventre--Pereira \cite{bonventre:19,BonventrePereira}.  The category $\mathcal{C}$ can be constructed as Grothendieck construction on the functor $G/H\mapsto \mathscr{F}_{G/H,*}^{\op}$, and its objects are pairs $(G/H,A_+)$ where $A_+$ is a retractive finite $G$-set over $G/H$.  Morally, such a result would amount to expressing the iterated Grothendieck construction of \cref{theresult} into a single Grothendieck construction. If the two Grothendieck constructions under consideration had been of the same type then such a construction would be formal, but as they are not we were unable to make this simplification work in this setting.
\end{remark}

\appendix
\section{Oplax functors and the Grothendieck construction} \label{sec:cat background}

Because the Grothendieck construction is key to many of our arguments in this paper, particularly the version for oplax functors, we recall the main definitions here; see also \cite[Chapter 10]{JohnsonYau2cats}.  We also include some results that are difficult to find in the literature. 

Let $\Cat$ denote the 2-category of small categories, functors and natural transformations.  We begin with the following preliminary definition.

\begin{definition} \label{defn:whiskering}
    Given a diagram of categories, functors and a natural transformation of the form
    \[\begin{tikzcd}
	{\mathcal A} \ar[r, "F"] & {\mathcal B} \ar[rr, bend left=30, "G"{name=U}]\ar[rr, bend right=30, swap, "G'"{name=L}] && {\mathcal C} \ar[r, "H"] & {\mathcal D,}
    \ar[Rightarrow, from=U, to=L, "\alpha", shorten <= 4pt, shorten >= 4pt]
    \end{tikzcd}\]
    the \emph{whiskering} of $\alpha$ by $F$ is the natural transformation denoted by $\alpha \cdot F \colon G\circ F \to G'\circ F$, whose component at the object $a$ of $\mathcal{A}$ is given by $\alpha_{F(a)} \colon GF(a) \to G'F(a)$.  Similarly, the \emph{whiskering} of $\alpha$ by $H$ is the natural transformation denoted by $H \cdot \alpha \colon H\circ G \to H\circ G'$, whose component at the object $b$ of $\mathcal{B}$ is given by $H(\alpha_{b}) \colon HG(b) \to HG'(b)$.
\end{definition}

While there is a more general definition between arbitrary 2-categories, here we define the notion of an oplax functor from a small category to  the 2-category $\Cat$.  Observe the role of whiskering in the definition, although we do not refer to it explicitly.

\begin{definition}\label{defn:oplax}
    Let $\mathcal{A}$ be a small category.  An \emph{oplax functor} $(F,\tau) \colon \mathcal{A}^{\op} \to \Cat$ consists of the following data:
    \begin{itemize}
        \item for every object $a$ in $\mathcal{A}$, a category $F(a)$;
        
        \item for every morphism $f\colon a \to b$ in $\mathcal{A}$, a functor $F(f)\colon F(b) \to F(a)$;
        
        \item for every pair of composable pairs of morphisms $a\xrightarrow{f} b \xrightarrow{g} c$ in $\mathcal{A}$, a natural transformation 
        \[ \tau_{f,g}\colon F(gf) \Rightarrow F(f)F(g);
        \]
    
        \item for every object $a$ of $\mathcal{A}$, a natural transformation
        \[\tau_a \colon F(\id_a) \Rightarrow \id_{F(a)}.\]
    \end{itemize}
    These data are subject to the following coherence conditions:
    \begin{itemize}
        \item for every composable triple of morphisms $a\xrightarrow{f} b \xrightarrow{g} c \xrightarrow{h} d$ in $\mathcal{A}$, the diagram of natural transformations 
        \[ \begin{tikzcd}
        F(hgf)\rar[Rightarrow, "\tau_{gf,h}"]\dar[Rightarrow,"\tau_{f,hg}"'] & F(gf)F(h)\dar[Rightarrow," \tau_{f,g}\cdot F(h)"] \\
        F(f)F(hg)\rar[Rightarrow, "F(f)\cdot \tau_{g,h}"'] & F(f)F(g)F(h)
        \end{tikzcd} \]
        commutes; and
        
        \item for every morphism $f\colon a \to b$ in $\mathcal{A}$, the diagrams
        \[ \begin{tikzcd}[column sep=small]
            F(f\circ \id_a) \ar[r, equal]\ar[d, Rightarrow,"\tau_{\id_a,f}"'] &  \id_{F(a)}\circ F(f)\\
             F(\id_a)\circ F(f)\ar[ur, Rightarrow, " \tau_a \cdot F(f)"']
        \end{tikzcd}\qquad 
        \begin{tikzcd}[column sep=small]
            F(\id_b\circ f) \ar[r, equal]\ar[d, Rightarrow,"\tau_{f,\id_b}"'] & F(f)\circ \id_{F(b)}\\
           F(f)\circ F(\id_b)\ar[ur, Rightarrow, "F(f)\cdot \tau_b"']
        \end{tikzcd} \]
        commute.
    \end{itemize}
\end{definition}

In many cases of interest the structure natural transformations $\tau_{f,g}$ and $\tau_a$ of an oplax functor are actually natural isomorphisms.  

\begin{definition}\label{def:pseudofunctor}
    An oplax functor $(F,\tau)\colon \mathcal{A}^{\op}\to \Cat$ is a  \emph{pseudofunctor} if for all objects $a$ of $\mathcal{A}$ and composable morphisms $f$ and $g$, the natural transformations $\tau_{f,g}$ and $\tau_a$ are both natural isomorphisms.
\end{definition}

Just as functors are compared using natural transformations, we now briefly recall the notion of a pseudonatural transformation, which is used to compare pseudofunctors.

\begin{definition}
    Let $(F,\tau),(G,\sigma)\colon \mathcal{A}^{\op}\to \Cat$ be two pseudofunctors.  A \emph{pseudonatural transformation} $\eta\colon (F,\tau)\Rightarrow (G,\sigma)$ consists of the following data:
    \begin{itemize}
        \item for every object $a$ of $\mathcal{A}$, a functor $\eta_a\colon F(a)\to G(a)$; and
        
        \item for every morphism $f\colon a\to b$ in $\mathcal{A}$, a natural isomorphism $\eta_f$ inhabiting the square of functors
        \[
        \begin{tikzcd}
            F(b) \ar[d,"F(f)"'] \ar[r,"\eta_b"] & G(b) \ar[d,"G(f)"]\\
            F(a) \ar[r,"\eta_a"'] & G(a). \ar[from = 1-2,to = 2-1,Rightarrow,"\eta_f",shorten=4mm]
        \end{tikzcd}
        \]
    \end{itemize}
    The functors $\eta_a$ and natural isomorphisms $\eta_f$ are subject to various compatibility conditions with the the structure maps of $(F,\tau)$ and $(G,\sigma)$.  We refer the reader to \cite[Section 4.2]{JohnsonYau2cats} for details.
 \end{definition}

We sometimes denote an oplax functor $(F, \tau)$ simply by its first coordinate $F$.  The idea here is that such an assignment $F$ on objects and morphisms fails to be a functor precisely because the natural transformations $\tau_{f,g}$ and $\tau_a$ are not identities.  Nonetheless, these data are sufficient to define the Grothendieck construction.

\begin{definition}\label{defn:grothconstruction}
    Given an oplax functor $(F,\tau) \colon \mathcal{A}^{\op}\to \Cat$, its \emph{Grothendieck construction} $\int_{\mathcal{A}^{\op}} F$ is the category in which:
    \begin{itemize}
        \item objects are pairs $(a,x)$ consisting of an object $a$ of $\mathcal{A}$ and an object $x$ of the category $F(a)$;
    
        \item morphisms $(a,x)\to (b,y)$ are pairs $(f,\alpha)$ consisting of a morphism $f\colon b\to a$ in $\mathcal{A}$ and a morphism $\alpha\colon F(f)(x)\to y$ in $F(b)$; 

        \item the composite of morphisms $(f,\alpha)\colon (a,x)\to (b,y)$ and $(g,\beta)\colon (b,y)\to (c,z)$ is given by $fg$ in the first coordinate and the composite
        \[ F(fg)(x)\xrightarrow{(\tau_{g,f})_x} F(g)(F(f) (x))\xrightarrow{F(g)(\alpha)} F(g)(y) \xrightarrow{\beta} z\]
        in the second coordinate; and

        \item the identity morphism on $(a,x)$ is given by the pair $(\id_a,(\tau_a)_x)$.
    \end{itemize}
\end{definition}

Checking that this composition is associative and unital amounts to using the naturality and coherence conditions on $\tau$. The category $\int_{{\mathcal A}^{\op}} F$ is equipped with a canonical projection $\int_{\mathcal{A}^{\op}} F\to \mathcal{A}^{\op}$ onto the first coordinate.

In some cases, one can also consider a slightly different version of the Grothendieck construction, in which the direction of the morphisms is reversed. Although (op)lax functors admit this construction as well, we include this definition here for pseudofunctors, which is all we need for the purposes of this paper.  Note that we make no distinction on the name and use ``Grothendieck construction'' for both, as the ambiguity is taken care of by the change in notation.

\begin{definition}\label{defn:other Groth construction}
    Given a pseudofunctor $F\colon \mathcal{A}^{\op}\to \Cat$, its \emph{Grothendieck construction}$\int^{\mathcal{A}^{\op}} F$ is the category in which:
    \begin{itemize}
        \item objects are pairs $(a,x)$ consisting of an object $a$ of $\mathcal{A}$ and an object $x$ of the category $F(a)$;
        
        \item morphisms $(a,x)\to (b,y)$ are pairs $(f,\alpha)$ consisting of a morphism $f\colon a\to b$ in $\mathcal{A}$ and a morphism $\alpha\colon x\to F(f)(y)$ in $F(a)$; 
        
        \item the composite of morphisms $(f,\alpha)\colon (a,x)\to (b,y)$ and $(g,\beta)\colon (b,y)\to (c,z)$ is given by $gf$ in the first coordinate and the composite
        \[ x\xrightarrow{\alpha} F(f)(y) \xrightarrow{F(f)(\beta)} F(f)(F(g)(z)) \xrightarrow{\tau^{-1}_{f,g}} F(gf)(z)\]
        in the second coordinate; and
        
        \item the identity morphism on $(a,x)$ is given by the pair $(\id_a,\id_x)$.
    \end{itemize}
    Similarly to the version in \cref{defn:grothconstruction}, the category $\int^{\mathcal{A}^{\op}} F$ is equipped with a canonical projection $\int^{\mathcal{A}^{\op}} F\to \mathcal A$. 
\end{definition}

We conclude this appendix with two well-known results that we use in \cref{sec:Gtree}.

\begin{proposition}\label{equivalence on Grothendiecks}
    Let $F, F'\colon I^{\op}\to \Cat$ be oplax functors and suppose there is an oplax natural transformation $\eta\colon F\to F'$ such that $\eta_i$ is an equivalence of categories for each object $i$ of  $I$. Then there is an equivalence of categories
    \[ \int_{I^{\op}} F \simeq \int_{I^{\op}} F'.  \]
\end{proposition}

\begin{proposition} \label{Grothendieck indexing cat change} 
    For any functor $G\colon J^{\op}\to I^{\op}$ and oplax functor $F\colon I^{\op}\to \Cat$, there is a functor 
    \[ \int_{G} F\colon \int_{J^{\op}} FG\to\int_{I^{\op}} F.\] 
    Moreover, if $G$ is an equivalence of categories, then so is $\int_{G} F$.
\end{proposition}

Observe that, while we are not assuming that the functor $\int_G F$ is a Grothendieck construction, we find using similar notation for it is suggestive of its behavior.

\begin{proof}
    Recall that an object of $\int_{J^{\op}} FG$ is given by a pair $(j, x)$ where $j$ is an object of $J$ and $x$ is an object of $FG(j)$.  We define $\int_GF$ by sending $(j,x)$ to $(G(j),x)$, and the morphism $(f,\alpha)\colon (j,x)\to (k,y)$ to $(G(f),\alpha)\colon (G(j), x)\to (G(k), y)$ in $\int_{I^{\op}} F$.  This assignment preserves identities because $G$ is a functor, and it preserves composition because $G$ is a functor and because $\int_{G}F$ is the identity on the second coordinate.  Observe this functor is full and faithful whenever $G$ is. 

    To prove the second statement, suppose that $G$ is an equivalence of categories.  If $(i,x)$ is an object in $\int_{I^{\op}} F$, choose an object $j$ of $J$ together with an isomorphism $\ell\colon i \to G(j)$ in $I$.  Then $F(\ell)\colon FG(j)\to F(j)$ is an isomorphism of categories, so there exists an object $y$ of $FG(j)$ such that $F(\ell)(y)=x$.  Then the map $(\ell,\id_{x})\colon (G(j),y)\to (i,x)$ is an isomorphism, so $\int_G F$ is essentially surjective.  Applying the observation about fully faithfulness above, we conclude that $\int_G F$ is an equivalence of categories.
\end{proof}

\bibliographystyle{alpha}
\bibliography{references}

\end{document}